\documentclass[french]{article} 
\usepackage{graphicx} 

\usepackage[utf8]{inputenc}
\usepackage[T1]{fontenc}
\usepackage{lmodern}
\usepackage[a4paper]{geometry}
\usepackage[english]{babel}

\usepackage{fullpage}

\usepackage{amssymb,amsmath,mathtools,amsthm,dsfont,stmaryrd, bm}
  
\usepackage[all]{xy}
\usepackage{todonotes}

\newtheorem{theorem}{Theorem}
\newtheorem{prop}[theorem]{Proposition}

\newtheorem{lemma}[theorem]{Lemma}

\newtheorem{definition}[theorem]{Definition}

\newcommand{\E}{\mathbb{E}}

\newcommand{\G}{\mathbb{G}}

\renewcommand{\L}{\mathbb{L}}

\newcommand{\N}{\mathbb{N}}

\renewcommand{\P}{\mathbb{P}}
\newcommand{\Q}{\mathbb{Q}}
\newcommand{\R}{\mathbb{R}}
\renewcommand{\S}{\mathbb{S}}

\newcommand{\V}{\mathbb{V}}

\newcommand{\Z}{\mathbb{Z}}

\newcommand{\cA}{\mathcal{A}}
\newcommand{\cB}{\mathcal{B}}
\newcommand{\cC}{\mathcal{C}}
\newcommand{\cD}{\mathcal{D}}
\newcommand{\cE}{\mathcal{E}}
\newcommand{\cF}{\mathcal{F}}
\newcommand{\cG}{\mathcal{G}}
\newcommand{\cH}{\mathcal{H}}

\newcommand{\cL}{\mathcal{L}}

\newcommand{\cP}{\mathcal{P}}

\newcommand{\cR}{\mathcal{R}}
\newcommand{\cS}{\mathcal{S}}
\newcommand{\cT}{\mathcal{T}}
\newcommand{\cU}{\mathcal{U}}
\newcommand{\cV}{\mathcal{V}}

\newcommand{\cY}{\mathcal{Y}}
\newcommand{\cZ}{\mathcal{Z}}

\newcommand{\fE}{\mathfrak{E}}
\newcommand{\fA}{\mathfrak{A}}

\renewcommand{\1}{\mathds{1}}
\renewcommand{\d}{\text{d}}

\renewcommand{\epsilon}{\varepsilon}
\renewcommand{\phi}{\varphi}

\DeclareMathOperator{\cone}{cone}
\DeclareMathOperator{\size}{size}

\begin{document}

\title{\LARGE The $K$-th nearest neighbor random walk on a Poisson point process gets trapped}

\author{Anne-Laure Basdevant\footnote{LPSM - UMR CNRS 8001, Sorbonne Universit\'e, 4 place Jussieu, 75005 Paris, France, {\it anne.laure.basdevant@normalesup.org}}, David Coupier\footnote{Institut Mines Télécom Nord Europe, Cité Scientifique, 59655 Villeneuve d’Ascq, France, {\it david.coupier@imt-nord-europe.fr}}, Jean-Baptiste Gou\'er\'e\footnote{Institut Denis Poisson - UMR CNRS 7013, Universit\'e de Tours, Parc de Grandmont, 37200 Tours, France, {\it jean-baptiste.gouere@univ-tours.fr}} ~and Marie Th\'eret\footnote{Modal'X, Universit\'e Paris Nanterre, 92000 Nanterre, France, {\it marie.theret@parisnanterre.fr}}}

\date{}

\selectlanguage{english}

\maketitle

\begin{abstract}
The $K$-th nearest neighbor random walk $(X_n)_{n \geq 0}$ on a homogeneous Poisson point process $\chi$ on $\R^d$ ($d\geq 1$), starts at the origin and at each step picks its next Poisson point among its closest neighbors according to i.i.d. labels having the same distribution as $K$. Our main result (Theorem \ref{t:V}) states that the number of Poisson points visited by $(X_n)_{n \geq 0}$ admits an exponential decay whenever the random variable $K$ has a bounded support (BS). In particular, the $K$-th nearest neighbor random walk visits finitely many Poisson points if and only if $K$ satisfies Assumption (BS). To prove it, we introduce the key notion of {\it pioneer point} which allows us to deal with the region of $\R^d$ already explored by $(X_n)_{n \geq 0}$. Still under Assumption (BS), we also prove an exponential decay for the Euclidean length of the trajectory performed by  $(X_n)_{n \geq 0}$ (Theorem \ref{t:L}). Finally, and quite surprisingly, we exhibit an example of label distribution with bounded support for which the $K$-th nearest neighbor random walk discovers new Poisson points after a number of steps whose tail distribution is at least polynomial (Theorem \ref{t:S}).
\end{abstract}

\section{Introduction and main results}

\subsection{Motivations}
\label{s:motivations}

Let $\chi$ be an homogeneous Poisson Point Process (PPP) on $\R^d$, $d \geq 1$, with intensity $1$ w.r.t. the Lebesgue measure $\cL^d$. We work under the Palm measure, \textit{i.e.}, the PPP $\chi$ is conditioned to contain the origin $0 \in \R^d$. Let $K$ be a random variable taking values in $\N^* \coloneqq \{1,2,\dots \}$, with probability distribution $\mu$. The \textit{$K$-th nearest neighbor random walk $(X_n)_{n \geq 0}$ on the process $\chi$} starts at the origin and, at each step, picks its next Poisson point according to i.i.d. labels having the same distribution as $K$. Precisely, for any $x \in \chi$ and for any $k \in \N^*$, let us denote by $v_k(x)$ the $k$-th nearest neighbor of $x$ in $\chi$ (w.r.t. the Euclidean distance). Hence, we set
\[
X_0 \coloneqq 0 \, \mbox{ and } \, \forall n \geq 1 , \, X_n \coloneqq v_{K_n} (X_{n-1}) ~,
\]
where  $(K_n)_{n\geq 1}$ is a family of i.i.d.\ random variables with common distribution $\mu$, which is independent of the PPP $\chi$. The random variable $K_n$ is called the {\it $n$-th label}.

The main question we want to address here is the following: which conditions on the distribution $\mu$ and the dimension $d$ (the two parameters of our model) allow the $K$-th nearest neighbor random walk to visit infinitely many Poisson points, or in an equivalent way, to go out any compact set? For that purpose, let $\V$ be the range of the random walk, {\it i.e.}, the set of points of $\chi$ that are visited by $(X_n)_{n\geq 0}$:
\[
\V \coloneqq \{ x \in \chi : \, \exists n \geq 0 , \, X_n = x \} ~.
\]
We say that the $K$-th nearest neighbor random walk \textit{gets trapped} when its range $\V$ is finite. It is not difficult to see that the set $\V$ will be a.s. infinite whenever the support of the probability measure $\mu$ is unbounded. So, we assume in what follows that $\mu$ has a bounded support (we write BS for short), meaning that its supremum $M$ is finite:
\[
M \coloneqq \sup \{ k \in \N^* :\, \mu (k) >0 \} < \infty \,. \qquad \qquad \textrm{(BS)} 
\]

Let us now introduce the directed graph $\G_\mu$ including all the possible moves of the $K$-th nearest neighbor random walk. Its vertex set is given by the points of $\chi$ and its edge set contains all the couples $(x,y)$ where $x,y \in \chi$ and $y = v_k(x)$ for some integer $k$ in the support of $\mu$. Let us denote by $\mathcal{C}_0$ the cluster of the origin, \textit{i.e.}, the set of Poisson points reached by a directed path of $\G_\mu$ starting at $0$. Note that the only
randomness in $\G_\mu$ comes from the position of the points in $\chi$. The undirected version of this graph has been studied by H\"aggstr\"om and Meester \cite{HaggstromMeester} in the case where the support of $\mu$ is exactly $\{1,\ldots,M\}$. Without major changes, their arguments (namely, the proofs of Theorems 2.1 and 2.2) apply to our directed context and lead to the following statement. In the case where the support of $\mu$ is exactly $\{1,\ldots,M\}$ and for any dimension $d \geq 2$, there exists a critical integer $2 \leq \vec{M}_c(d) < \infty$ such that: for any $M < \vec{M}_c(d)$, the cluster $\mathcal{C}_0$ is a.s. finite while for any integer $M \geq \vec{M}_c(d)$, it is infinite (and then unbounded) with positive probability. So, in the subcritical phase, the cluster $\mathcal{C}_0$ being finite, there is no way for the $K$-th nearest neighbor random walk $(X_n)_{n \geq 0}$ to escape to infinity. It will get trapped with probability $1$. But this is in the supercritical case that our initial problem really makes sense. Indeed, in that case, $\mathcal{C}_0$ contains with positive probability some infinite directed paths starting from the origin that the $K$-th nearest neighbor random walk could use to escape to infinity and visit infinitely many Poisson points.

Consider the special case where $\mu$ is a Dirac measure. In this case, the only randomness comes from the PPP, and everything is encoded in the graph $\G_\mu$. More specifically, the range $\V$ is infinite if and only if there exists an infinite (directed) path starting from $0$ in $\G_\mu$, that is, if and only if the forward cluster of $0$ in $\G_\mu$
is infinite. The case $\mu=\delta_1$ is considered in \cite{HaggstromMeester} where it is shown that there is no infinite path starting from $0$ in $\G_\mu$ (such paths are called infinite descending chains in \cite{Daley-Last-05}). 
The same result is proven in the general case $\mu = \delta_k$ , for any $k \in \N^*$, in \cite{LeStum}. Finally, let us mention that in this Dirac setting, the graph $\G_\mu$ is an out-degree-one graph defined in a deterministic and translation-equivariant way from the PPP.  This kind of graph has the advantage of possessing strong structural properties and appears in various models as for instance systems of stopped paths \cite{CDG,CDLS}.

\medskip

Our main result (Theorem \ref{t:V}) asserts that Assumption (BS) appears to be also a sufficient condition for the $K$-th nearest neighbor random walk to get trapped, whatever the dimension $d$. Precisely, for any $d \geq 1$, the range $\V$ is a.s. finite if and only if the distribution $\mu $ has a bounded support. In particular, when $\mu$ has a bounded support and generates with positive probability an infinite cluster $\mathcal{C}_0$ in the directed graph $\G_\mu$, the $K$-th nearest neighbor random walk fails to take advantage of the infinite directed paths contained in $\mathcal{C}_0$. Actually, Theorem \ref{t:V} goes further and says that the cardinality of the range $\V$ admits an exponential decay under Assumption (BS).

\subsection{Main results}

Recall that we work under the Palm measure. See \cite{last_penrose_2017} for a general reference on point processes and Palm measures. For an at most countable set $\cA$, we denote by $|\cA|$ its cardinality, that may be infinite. Here is our main result:

\begin{theorem}[Number of visited points]
\label{t:V}
Let $d \geq 1$ and suppose that Assumption (BS) is satisfied. Then there exist two constants $C_1, C_2 \in (0,+\infty)$ depending\footnote{In fact we obtain constants depending more precisely on $d$, $M$ and $\mu (\{M\})$.} on the dimension $d$ and on the probability distribution $\mu$ such that, for any integer $k \geq 1$,
$$
\P \big[ |\V| \geq k \big] \leq C_1 e^{-C_2 k} ~.
$$
In particular, the following equivalence holds for any $d \geq 1$: $Assumption \,\,(BS) \iff \V \mbox{ is a.s. finite}$.
\end{theorem}

Let us sketch the proof of Theorem \ref{t:V}. The main difficulty of the proof stems from the fact that, as it evolves, the random walk gathers information about its environment by revealing 
parts of the PPP.  This lack of independence leads us to introduce the notion of a {\em pioneer point} (see Definition \ref{d:pioneer}).  Informally, the position $x=X_n$ of the random walk at time $n$ 
is a pioneer point if it is the first time the random walk visits $x$, and if, at time $n+1$, the random walk can jump, with probability bounded away from $0$, into some ball $B$ of fixed 
radius in which the PPP remains totally unrevealed. Within this ball, again with probability bounded away from $0$, the configuration of the PPP is such that the random walk cannot escape 
from it. Thus, by the independence properties of the PPP, each time the random walk visits a pioneer point, there is a probability bounded away from $0$ that it becomes trapped at the 
next step, independently of the past (Lemma \ref{l:trap}). Therefore, the number of pioneer points has a subexponential tail (Lemma \ref{l:trap2}).

It then remains to prove that, with high probability, a macroscopic part of points visited by the $K$-th nearest neighbor random walk are pioneer points (see Lemma \ref{l:nice2}). To do it, we partition $\R^d$ into large boxes and introduce the notion of {\it good box} (see Definition \ref{d:good}) in such a way that each time the random walk approaches a good box for the first time, it must visit a pioneer point. Now, we can tune our parameters so that it becomes very likely for a box to be good and then the set of good boxes percolates in some strong sense. This forces the random walk to encounter many good boxes and so many pioneer points.

\medskip

Closely related to the number of visited points $|\V|$, we are also interested in the Euclidean length of the trajectory of the $K$-th nearest neighbor random walk $(X_n)_{n\geq 0}$ defined as
\[
\L \coloneqq \sum_{ \{x,y\} \subset \chi }  \| x-y \| \1_{\{ \exists n\geq 0 \,,\, \{X_n,X_{n+1}\} = \{x,y\} \}}
\]
where $\|\cdot\|$ denotes the Euclidean norm on $\R^d$. Notice that in the above definition, the length of each segment $[x,y]$ is counted at most once in the sum, even if the random walk $(X_n)_{n\geq 0}$ goes from $x$ to $y$ or from $y$ to $x$ several times. Note also that, by Theorem \ref{t:V}, $\L$ is almost surely finite if and only if Assumption (BS) holds. Under Assumption (BS), we actually prove that the distribution of $\L$ has an exponential tail.

\begin{theorem}[Length of the trajectory]
\label{t:L}
Let $d \geq 1$ and suppose that Assumption (BS) is satisfied. Then there exist two constants $C_3, C_4 \in (0,+\infty)$ depending\footnote{Similarly, we obtain constants depending more precisely on $d$, $M$ and $\mu (\{M\})$.} on the dimension $d$ and on the probability distribution $\mu$ such that, for any integer $k \geq 1$,
$$
\P \big[ \L \geq k \big] \leq C_3 e^{-C_4 k} ~.
$$
\end{theorem}

Thanks to Theorem \ref{t:V}, the proof of Theorem \ref{t:L} boils down to show that, with high probability, the random walk cannot have a long trajectory, say $\L \geq A k$ for some large constant $A$, while visiting not too many vertices, say $|\V| \leq k$. 
The proof proceeds as follows. We associate with each trajectory of the random walk a spanning tree of $\V$. By construction, the tree is a subgraph of the
graph $\G_\mu$ defined in Section \ref{s:motivations}. In particular, if $x$ is the parent of $y$ in the tree, then $y$ belongs to $B(x,r_M(x))$,  the closed Euclidean ball with center $x$ and radius $r_M(x) = \| x - v_M(x) \|$ (the distance between $x$ and its $M$-th closest Poisson point).
We then take a (continuous) union bound over all these trees. To control the complexity of the union bound, we need to show that the balls $B(x,r_M(x))$ have limited overlap.
This latter fact is guaranteed by a geometric result (known in the literature as the Stone's lemma \cite{Devroye-Gyorfi-Lugosi} and recalled here in Lemma \ref{l:stone}) asserting that the number of balls $B(x,r_M(x))$, with $x \in \chi$, containing a given $y \in \R^d$ is bounded above by a deterministic constant depending only on $M$ and $d$.

\medskip

To complete our study, we also focus on the number of steps needed by the random walk to discover its whole range $\V$:
\[
\S \coloneqq \inf \{ n\in \N \,:\, \V = \{ X_0, \dots , X_n\} \} ~.
\]
When the label distribution $\mu$ is the Dirac measure at $k$, for some $k\in\N^*$, we have $\S=|\V|-1$. Hence, by Theorem \ref{t:V}, $\S$ has an exponential tail.
However, we prove that this behavior does not extend to arbitrary distribution $\mu$ satisfying Assumption (BS) by exhibiting a simple example of a label distribution $\mu$ for which the tail distribution of $\S$ decays at least polynomially. 
  
\begin{theorem}[Number of steps to discover the whole range]
\label{t:S}
Let $d\geq 1$ and assume that $\mu \coloneqq  (1-p) \delta_1 + p \delta_2$, for a given $p\in (1/2,1)$. Then there exist two constants $C_5, C_6 \in (0,+\infty)$ depending on the dimension $d$ and the parameter $p$ such that, for any $k \in \N^\ast$,
\[
\P \big[ \S \geq k \big] \geq C_5 k^{-C_6} ~.
\]
\end{theorem}

The proof of Theorem \ref{t:S} consists in constructing an event on which $\V$ is an explicit set 
\[
\{0=z_0,z_1,\dots,z_{3L}\}
\]
whose exploration by the random walk is very slow.
We refer to Figure \ref{fig:bonenvironnement}. 
On this event, the point process restricted to some box only contains $3L+1$ points - namely the points of the set mentioned above - arranged in an almost aligned configuration with suitably chosen spacings.
When the random walk is in $z_i$ for some $i \in \{1,\dots, 3L-1\}$, it jumps to the left ($z_{i-1}$) or to the right ($z_{i+1}$) with a mean efficient bias to the left (see Proposition \ref{prop:defY} for a precise statement).
The exploration of $\V$ is therefore long. More precisely,
on this event of probability at least $\exp(-cL)$, the time needed to explore all $\V$ is at least $\exp(c'L)$. This yields the result.

\subsection{Some related works}

We provide in this section a brief overview of related models involving random walks on Poisson point processes, or on the underlying graph along which they move.

\paragraph{The $(k_1,k_2)$-nearest neighbor graph.}
In \cite{JT}, B. Jahnel and A. Tóbiás consider undirected graphs with degrees bounded by two, constructed from a stationary point process on $\R^d$ via equivariant edge-drawing rules.  The point process $X$ is assumed to satisfy mild assumptions, namely finite intensity, deletion tolerance, and nonequidistant property. They prove (see the proof of Theorem 2.5 in \cite{JT}) that such graphs a.s.\ contain only finite connected components (in other words, they do not percolate), provided that they are edge-preserving. The edge-preserving property means that  any edge $\{x,y\}$ of the graph remains present after removing any subset of vertices of $X\setminus\{x,y\}$. See Definition 3.2 in \cite{JT} for details. As observed by the authors (see Example 4.3 in \cite{JT}) the $(k_1,k_2)$-nearest neighbor graph (we give the definition below) does not satisfy the edge-preserving property, and then does not fall within the scope of their result. 

When the underlying process is a homogeneous PPP $X$ on $\R^d$, Theorem \ref{t:V} yields the absence of percolation in the 
$(k_1,k_2)$-nearest neighbor graph. Let us define the graph and then provide the proof. Let $0 < k_1 < k_2$ be two integers. The $(k_1,k_2)$-nearest neighbor graph $G_{k_1,k_2}(X)$ is an undirected graph whose vertex set is given by $X$ and edge set is defined as follows. For any $x,y \in X$, $\{x,y\}$ is an edge of $G_{k_1,k_2}(X)$ if and only if $x \in \{v_{k_1}(y),v_{k_2}(y)\}$ and $y \in \{v_{k_1}(x),v_{k_2}(x)\}$. We point out that $G_{k_1,k_2}(X)$ is not edge-preserving whenever $(k_1,k_2) \not= (1,2)$.

\begin{prop}
    \label{prop:JT}
    Let $X$ be a homogeneous PPP on $\R^d$.
    For any integers $0 < k_1 < k_2$, the $(k_1,k_2)$-nearest neighbor graph $G_{k_1,k_2}(X)$ a.s.\ contains only finite connected components.
\end{prop}

\begin{proof}
    We first follow the argument of \cite{JT}, Section 3.
    As the degree of each vertex is at most two,
    any unbounded component of $G_{k_1,k_2}(X)$ is of one of two types: either it contains  a single vertex of degree $1$ and is isomorphic to $\N$, or it contains no vertex of degree $1$ and is isomorphic to $\Z$.
    By a standard application of the mass-transport principle, the former case can be ruled out.
    Therefore,
    \begin{equation}\label{e:infinite-implies-Z}
        \text{all infinite clusters are bi-infinite chains of vertices of degree two.}
    \end{equation}
    
    At this stage, B. Jahnel and A. Tóbiás argue by contradiction, assuming that $G_{k_1,k_2}(X)$ contains an infinite connected component $\mathcal{C}$ with positive probability and thus with probability one. They then use in their context the edge-preserving property of the graph to break $\mathcal{C}$ into several pieces, one of which must be an infinite connected component having a vertex of degree one, which is forbidden by the previous argument. As explained above, this argument does not apply when $(k_1,k_2) \neq (1,2)$. 
    
    Instead, we proceed as follows. 
    Let $K$ be a random variable taking the values $k_1$ and $k_2$, each with probability one half. 
    Consider the $K$-th nearest neighbor random walk on $X$ starting from a given point $x_0$ of $X$ and denote by $\V(x_0,X)$ its range.
    By Theorem \ref{t:V}%
    \footnote{
        Let us keep the notation $\V$ for the range of the $K$-th nearest neighbor random walk on $\chi$ (the PPP under the Palm measure) started from $0$. 
        By Theorem \ref{t:V}, $\V$ is almost surely finite. Therefore
        \[
            0=\P[\V \text{ is infinite}] = \E\left[\sum_{x \in X \cap [0,1]^d} \1_{\{\V(x,X) \text{ is infinite}\}}\right].
        \]
        Thus, $\V(x,X)$ is almost surely finite for any $x \in X \cap [0,1]^d$.
        By stationarity, the same holds for any $x \in X$.
    },
    \begin{equation} \label{e:V-finite}
        \text{for all $x_0$ in $X$, $\V(x_0,X)$ is finite.}
    \end{equation}

    We now prove
    \begin{equation} \label{e:degree-2-V=C}
        \text{for all $x_0$ in a degree-$2$ connected component $\cC$ of $G_{k_1,k_2}(X)$, $\cC=\V(x_0,X)$}.
    \end{equation}
    Proposition \ref{prop:JT} is a consequence of \eqref{e:infinite-implies-Z}, \eqref{e:V-finite} and \eqref{e:degree-2-V=C}.
    Let $x_0$ and $\cC$ be as in \eqref{e:degree-2-V=C}. 
    Since any vertex of $\cC$ has degree $2$, the $k_1$-th and $k_2$-th nearest neighbor of any point of $\cC$ also belong to $\cC$. Hence, the $K$-th nearest neighbor random walk started from any $x_0 \in \mathcal C$ remains in $\mathcal C$ and defines an irreducible Markov chain on $\mathcal C$.  So, its range $\V(x_0,X)$ is finite  if and only if   $\mathcal{C}$ is finite and then $\mathcal{C}=\V(x_0,X)$.
\end{proof}

\paragraph{The greedy walk problem.} Consider a homogeneous PPP $\chi$ under the Palm measure. 
A greedy random walk starts at the origin and, at each step, deletes the point at its position and moves to the closest remaining point of $\chi$.
We refer to \cite{BFL-Greedy} for a review.
One of the main questions about greedy random walks is whether all points of $\chi$ are eventually deleted. 
In dimension $d=1$, it is easy to prove that the answer is negative.
The problem remains however open for dimension $d\ge 2$.

Let us highlight a result in dimension $d=1$ for a variant of the model in which some points of $\chi$
are double points, meaning that the greedy walk has to visit them twice to delete them.
More precisely, each point is a double point with probability $p>0$, independently of everything else.
The model is designed to mimic the behavior of the original model on $\R\times[0,\epsilon]$.
The authors of \cite{RST-Greedy} prove that this modified greedy walk deletes all the points.

\paragraph{Other random walks on Poisson point processes.} 
Let us mention some other related works concerning random walks on a point process in $\R^d$. In \cite{CFG}, jumps between two points $x, y$ of the point process occur with probability proportional to $\varphi(|x-y|)$ for some appropriate function $\phi$. The resulting random walk is thus reversible. The authors establish transience and recurrence results depending on the parameters of the model.

In \cite{Rousselle-recurrence}, \cite{Rousselle-quenched} and \cite{Rousselle-annealed} the author investigates random walks on graphs, such as the Delaunay graph, associated with point processes with suitable properties, such as Poisson point processes. He proves recurrence and transience criterion as well as quenched and annealed invariance principles.

\subsection{Notations, organisation of the paper}

The space $\R^d$ is equipped with the Euclidean norm $\| \cdot \|$. Occasionally, we also consider the $\infty$-norm $\| \cdot\|_\infty$ defined by $\| x \|_\infty = \sup \{ |x_i| \,,\, i \in \{ 1,\dots , d  \} \}$ for any $x \in \R^d$ with coordinates $(x_1, \dots , x_d)$. We denote by $d_2$ (resp. $d_{\infty}$) the distance associated with the Euclidean norm $\| \cdot \|$ (resp. the $\infty$-norm $\| \cdot\|_\infty$). Let us denote by $B(x,r)$ the closed Euclidean ball with center $x\in \R^d$ and radius $r\geq 0$, and by $S(x,r) = \partial B (x,r)$ the corresponding Euclidean sphere. Given $x \in \chi$ and $k \in \N^*$, recall that $v_k(x)$ is the $k$-th nearest neighbor of $x$ in $\chi$ (w.r.t. the Euclidean distance) that is defined without any ambiguity with probability $1$. Let also $r_k (x)$ be the Euclidean distance between $x$ and its $k$-th nearest neighbor in $\chi$, i.e., $r_k(x) \coloneqq \| x- v_k(x) \|$.

\medskip

The structure of the paper is straightforward. Section \ref{s:V} is devoted to the proof of Theorem \ref{t:V} which appears as an immediate consequence of Lemmas \ref{l:trap2} and
\ref{l:nice2}. In Section \ref{s:Pioneer}, the key notion of {\it pioneer point} is introduced (Definition \ref{d:pioneer}). Section \ref{s:L}, in which the Stone's lemma is recalled (Lemma \ref{l:stone}), is devoted to the proof of Theorem \ref{t:L}. Section \ref{s:S} is devoted to the proof of Theorem \ref{t:S}. The special Poissonian environment is described in Section \ref{s:FavorableEnv} and its probability is lower bounded in Lemma \ref{l:probaH}. Actually, proofs of Theorems \ref{t:V} and \ref{t:L} only work in dimension $d \geq 2$. The particular (and much easier) case of dimension $d = 1$ is treated separately in Section \ref{s:dim1}.

\section{Dimension 1}
\label{s:dim1}

Several arguments used to prove Theorems \ref{t:V} and \ref{t:L} in dimension $d \ge 2$ fail when $d = 1$. However, in this case, more elementary arguments allow us to establish both results. We present them in this short section, starting with the following lemma.

\begin{lemma} Assume that Assumption (BS) holds.
Let $(X_n)_{n \ge 1}$ be a $K$-th nearest neighbor random walk in dimension $d=1$. 
Then there exists a constant $\kappa_0 \in (0,+\infty)$, depending only on $M$, such that for any $k \in \N^\ast$,
\[
\P \big[ \sup_{n \ge 0} |X_n| \ge k \big] \le 2 e^{-\kappa_0 k}.
\]
\end{lemma}

\begin{proof}
By symmetry, 
\[
\P \big[ \sup_{n \ge 0} |X_n| \ge k \big] \le 2 \,\P\big( \sup_{n \ge 0} X_n \ge k \big).
\]
For $k \ge 0$, let us denote by $\mathcal{A}_k$ the event
\[
\mathcal{A}_k := \{ |\chi \cap (3k,3k+1)| = 0,\ |\chi \cap (3k+1,3k+2)| = M+1,\ |\chi \cap (3k+2,3k+3)| = 0 \}.
\]
The events $(\mathcal{A}_k)_{k \ge 0}$ are independent and satisfy
$
\P(\mathcal{A}_k) = \frac{e^{-3}}{(M+1)!}.
$
Moreover, on the event $\mathcal{A}_k$, the walk cannot cross the interval $(3k+2,3k+3)$, and hence cannot reach Poisson points that are located on $ (3k+2,+\infty)$. Therefore,
\[
\P\big( \sup_{n \ge 0} X_n \ge 3k+2 \big) \le \left(1 - \frac{e^{-3}}{(M+1)!} \right)^k.
\]
This yields the desired exponential bound.
\end{proof}

We are now in a position to prove Theorems \ref{t:V} and \ref{t:L} in dimension $d=1$.

\begin{proof}[Proof of Theorem \ref{t:V} in dimension $d=1$.]
Let $\V$ be the set of points of $\chi$ visited by the walk. We write
\[
\P(|\V| \ge k) \le \P\big( \sup_{n \ge 0} |X_n| \ge k/4 \big) 
+ \P\big( |\chi \cap [-k/4,k/4]| \ge k \big).
\]
The first term is exponentially small in $k$ by the previous lemma. 
For the second term, we note that under the Palm measure $\P$, the random variable $|\chi \cap [-k/4,k/4]| - 1$ follows a Poisson distribution with parameter $k/2$. Standard large deviation estimates for Poisson random variables then imply that
$\P\big( |\chi \cap [-k/4,k/4]| \ge k \big)
$
also decays exponentially fast in $k$. 
\end{proof}

\begin{proof}[Proof of Theorem \ref{t:L} in dimension $d=1$.]
Let $\L$ be the length of the trajectory, as defined in the introduction. From any vertex $x \in \chi$, the walk can jump to  $M$  different neighbors, each at  distance at most $r_M(x)$. Hence,
\[
\L \le \sum_{x \in \V} M\, r_M(x).
\]
Let $(x_i)_{i \in \Z}$ be the points of $\chi$, indexed so that $x_0 = 0$ and $x_i < x_{i+1}$. Define $\xi_i := x_i - x_{i-1}$. Then $(\xi_i)_{i \in \Z}$ are i.i.d. exponential random variables with parameter $1$, and
\[
r_M(x_i) \le \sum_{j=i-M+1}^{i+M} \xi_j.
\]
Moreover, if $x_i \in \V$, then necessarily $|i| \le M |\V|$. Combining these observations, we obtain
\begin{align*}
\P(\L \ge A k) 
&\le \P(|\V| \ge k) + \P\Big( \sum_{i=-Mk}^{Mk} M\, r_M(x_i) \ge A k \Big) \\
&\le \P(|\V| \ge k) + \P\Big( \sum_{j=-(M+1)k}^{(M+1)k} 2M^2\, \xi_j \ge A k \Big).
\end{align*}
Choosing $A > 4M^2(M+1)$ ensures that the second term decays exponentially fast in $k$ by standard large deviation estimates for sums of i.i.d. exponential random variables. 
\end{proof}

%%%%%%%%%%%%%%%%%%%%

\section{Number of visited points}
\label{s:V}

\subsection{Pioneer points}
\label{s:Pioneer}

Let us start with some notations. For any $n\in \N^*$, we denote by $E_n$ the random subset of $\R^d$ defined by
$$
E_n \coloneqq \bigcup_{k=0}^{n-1} B (X_k, r_{M} (X_k)) \,.
$$
Notice that the part of the space $\R^d$ that has already been explored by the random walk $(X_k)_{k\geq 0}$ until the $n$-th step is exactly $\bigcup_{k=0}^{n-1} B (X_k, r_{K_{k+1}} (X_k))$, which is included in $E_n$, uniformly in $(K_k)_{k\geq 1}$ (but the trajectory of the random walk does depend on the random labels). We define $E_0 \coloneqq \emptyset$. For any $n\in \N^*$, we denote by $\cG_n$ the $\sigma$-algebra $\cG_n \coloneqq \sigma ((K_i)_{i=1,\dots , n} ; \chi\vert_{E_n})$ corresponding to the maximal quantity of information that can be gathered by the exploration process until the $n$-th step.

We introduce the notion of pioneer points. Roughly speaking, $(\varepsilon, A)$-pioneer points are points $x\in \chi$ that are reached by the random walk $(X_n)_{n\geq 0}$ in such a way that in a small ball of radius smaller than $A$ around $x$, there exists an unexplored region of size at least $\varepsilon$ in which the random walk could possibly jump at its next step.

\begin{definition}
\label{d:pioneer}
For any $\varepsilon >0$, $A>0$ and $n\in \N$, $X_n$ is called an $(\varepsilon, A)$-pioneer point if there exist some $\rho \in [0,A]$ and some $c\in \R^d$ such that 
\begin{itemize}
\item[$(i)$] $B(c,\varepsilon) \subset B(X_n, \rho) \cap (\R^d \setminus E_n)$;
\item[$(ii)$] $| (\chi \smallsetminus \{X_n\})\cap \left(   B(X_n, \rho) \cap E_n  \right) | < M$.
\end{itemize}
The set of $(\varepsilon, A)$-pioneer points is denoted by
$$
\V_p (\varepsilon, A) \coloneqq \{ x\in \chi \,:\, \exists n \in \N \,,\, X_n =x \textrm{ and $X_n$ is $(\varepsilon, A)$-pioneer}\} \,.
$$
\end{definition}

\begin{figure}[!ht]
\begin{center}
\includegraphics[width=10cm,height=4.3cm]{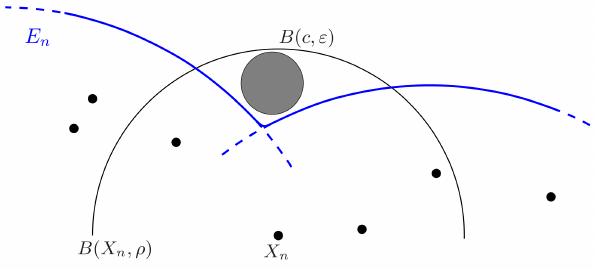}
\caption{\label{fig:pioneer1} Here is an example of a pioneer point $X_n$ illustrating Definition \ref{d:pioneer}. Indeed, the ball $B(X_n,\rho)$ exceeds the region $E_n$ (whose boundary is in blue) so that we can place a small (gray) ball $B(c,\varepsilon)$ within $B(X_n,\rho) \setminus E_n$ (Item $(i)$). Moreover, $B(X_n, \rho) \cap E_n$ only contains three Poisson points, except $X_n$, ensuring that Item $(ii)$ holds whenever the supremum of the support of the label distribution $\mu$ satisfies $M \geq 4$.}
\end{center}
\end{figure}

Let us comment the notion of pioneer point. It is implicit in the previous definition that if $X_n$ is a pioneer point then the vertex $x \in \chi$ such that $X_n = x$, is visited by the random walk for the first time at step $n$. Indeed, assume that the random walk visits some vertex $x \in \chi$ twice, say at steps $n$ and $m$ with $n < m$. In this case, the region $E_m$ includes the ball $B (X_n, r_{M} (X_n))$, also equal to $B (X_m, r_{M} (X_m))$, and then $X_m$ cannot satisfy Item $(ii)$. As a consequence, if $X_n$ and $X_m$, with $n \not= m$, are two pioneer points then they are necessarily different.

Another elementary but crucial remark about pioneer points is the following. The fact that $X_n$ is a pioneer point does not depend on the Poisson point process $\chi$ outside the region $E_n$ nor on the label $K_{n+1}$ of $X_n$. In other words, for any $\varepsilon, A >0$ and $n \in \N$, the event
\begin{equation}
\label{Pioneer-Gn-measurable}
\{ X_n \textrm{ is $(\varepsilon, A)$-pioneer}\} \mbox{ is $\cG_n$-measurable.}
\end{equation}

\paragraph{Sketch of the proof of Theorem \ref{t:V}.} Pioneer points play a crucial role in this proof. Indeed, if $|\V |$ is large, two events can occur: $(i)$ the number of pioneer points along the trajectory of the $K$-th nearest neighbor random walk $(X_n)_{n\geq 0}$ is large too, or $(ii)$ the number of pioneer points is small. The probability of $(i)$ is controlled by Lemma \ref{l:trap2} which is proved in Section \ref{s:min}. Its proof is based on the fact that each time the random walk reaches a pioneer point, the probability that its next step makes it fall into a trap from which it cannot escape is uniformly lower bounded by some positive constant. Hence the number of pioneer points reached by the random walk is stochastically dominated by a geometric distribution. The probability of $(ii)$ is controlled by Lemma \ref{l:nice2}. To prove this lemma, we first perform a renormalization argument in Section \ref{s:renorm}, introducing a notion of good blocks that are typical: thus if $|\V |$ is large, the trajectory of $(X_n)_{n\geq 0}$ has to reach a large number of good blocks. Then we show in Section \ref{s:nice} that each time the trajectory of the random walk gets close to a good block for the first time, it reaches, in fact, a pioneer point.

\begin{lemma}
\label{l:trap2}
For any $\varepsilon,A>0$, there exist positive constants $ \kappa_1, \kappa_2$ depending on parameters $\varepsilon, A, d, M,$ $\mu(\{M\})$ such that for any $k\in \N^*$, we have
\[
\P \big[ | \V_p (\varepsilon, A) | \geq k \big] \leq  \kappa_1 e^{-  \kappa_2 k}  \,.
\]
\end{lemma}

\begin{lemma}
\label{l:nice2}
There exist positive constants $\varepsilon_0$, $A_0$, $\kappa_3$,  $ \kappa_4$, $ \kappa_5$  depending on parameters $d, M$  such that for any $k\in \N^*$, we have
$$ \P [ | \V | \geq k \,,\, | \V_p (\varepsilon_0, A_0) | <   \kappa_3 k ] \leq  \kappa_4 e^{-  \kappa_5 k}  \,.$$
\end{lemma}

\begin{proof}[Proof of Theorem \ref{t:V}]
Let $\varepsilon_0= \varepsilon_0(d,M)>0, A_0= A_0 (d,M)>0,  \kappa_3 =  \kappa_3 (d,M) >0,  $ as given in Lemma \ref{l:nice2}. For any $k\in \N^*$, we have
\begin{equation}
\label{e:cclV1}
\P  [ | \V | \geq k ] \leq  \P [ | \V_p (\varepsilon_0, A_0) | \geq  \kappa_3 k ] + \P [ | \V | \geq k \,,\, | \V_p (\varepsilon_0, A_0) | <   \kappa_3 k ]\,.
\end{equation}
Let $ \kappa_1 =  \kappa_1 (d, M, \mu(\{M)\})$ and $ \kappa_2 =  \kappa_2 (d, M, \mu(\{M\}))$ as given by Lemma \ref{l:trap2} for these fixed $\varepsilon_0$ and $A_0$. By Lemma \ref{l:trap2}, for any $k\in \N^*$, we have
\begin{equation}
\label{e:cclV2}
 \P [ | \V_p (\varepsilon_0, A_0) | \geq  \kappa_3 k ] \leq  \kappa_1 e^{-  \kappa_2  \kappa_3 k}\,.
 \end{equation}
By Lemma \ref{l:nice2}, for any $k\in \N^*$, we have
\begin{equation}
\label{e:cclV3}
\P [ | \V | \geq k \,,\, | \V_p (\varepsilon_0, A_0) | <   \kappa_3 k ] \leq  \kappa_4 e^{-  \kappa_5 k}\,.
 \end{equation}
Combining Equations \eqref{e:cclV1}, \eqref{e:cclV2} and \eqref{e:cclV3}, Theorem \ref{t:V} is proved.
\end{proof}

\subsection{Pioneer points may lead to traps}
\label{s:min}

The goal of this section is to prove Lemma \ref{l:trap2} which asserts that it is very unlikely for the $K$-th nearest neighbor random walk to visit many (necessarily different) pioneer points.

Let us recall that the random walk is said to {\it get trapped} if $|\V | < \infty$. This is equivalent to the existence of a non-empty finite subset $\V' \subset \V$ which is closed for the $K$-th nearest neighbor random walk, {\it i.e.}, from which it cannot escape. Let us now define the notion of \textit{trap} as a particular example of closed subset for the $K$-th nearest neighbor random walk.

\begin{definition}
\label{d:trap}
A finite subset $\cB \subset \chi$ is called a trap if
\[
\forall x \in \cB , \, \big\{ v_k(x) , k\in \{1,\dots,M \} \big\} \subset \cB \,.
\]
The cardinality of the trap $\cB$ is called its size and denoted by $|\cB|$.
\end{definition}

Obviously, any trap has a size at least $M+1$. The next result states that the probability for the $K$-th nearest neighbor random walk of falling in a trap in one step, from any pioneer point, is uniformly bounded away from $0$.

\begin{lemma}
\label{l:trap}
For any $A,\varepsilon >0$, there exists $q = q(\varepsilon, A, d, M) >0$ such that, for any integer $n\in \N$, a.s. on the event $\{ X_n \textrm{ is $(\varepsilon, A)$-pioneer}\}$, the following holds:
\[
\P \big[ v_M (X_n) \textrm{ belongs to a trap of size }M+1 \,\big|\, \cG_n \big] \geq q \,.
\]
\end{lemma}

\begin{proof}
For the whole proof, we work conditionally to the $\sigma$-algebra $\cG_n$ with assuming that the event $\{ X_n \textrm{ is $(\varepsilon, A)$-pioneer}\}$ holds. Recall that this event is $\cG_n$-measurable (\ref{Pioneer-Gn-measurable}). Hence, there exist $\rho \in [0,A]$ and $c\in \R^d$ satisfying Items $(i)$-$(ii)$ of Definition \ref{d:pioneer}. By Item $(ii)$, there are $\ell \leq M-1$ Poisson points $y_1,\dots, y_{\ell}$ in $(\chi \smallsetminus \{X_n\}) \cap ( B(X_n, \rho) \cap E_n)$. Let us denote by $\cS$ the union of spheres centered at $X_n$ and passing through these $\ell$ points:
\[
\cS = \bigcup_{i=1}^{\ell} S(X_n,\|X_n-y_i\|)
\]
where we recall that $S(x,r)$ stands for the Euclidean sphere centered at $x$ with radius $r$. It is possible that the set $\cS$ intersects the ball $B(c,\varepsilon)$. Our goal is now to exhibit a small ball inside $B(c,\varepsilon)$ that avoids $\cS$ and within which we will be able to create a trap. To do it, let us consider $M$ (small) balls $B(c_i,\varepsilon/M)$ within $B(c,\varepsilon)$ whose centers $c_1,\ldots,c_M$ are aligned with $X_n$ and $c$, see Figure \ref{fig:pioneer2}. Since $\ell \leq M-1$, the Pigeonhole principle ensures that at least one of those balls does not overlap $\cS$, say $B(c_{i_0},\varepsilon/M)$.

\begin{figure}[!ht]
\begin{center}
\includegraphics[width=9cm,height=4.7cm]{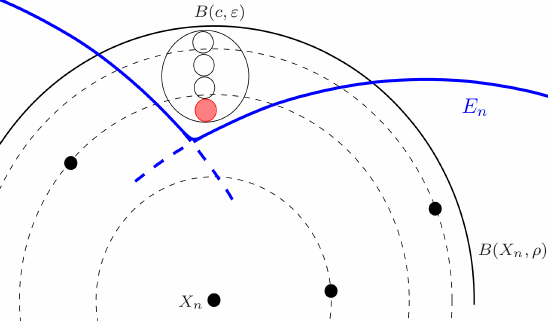}
\caption{\label{fig:pioneer2} Let us return to the example depicted in Figure \ref{fig:pioneer1} with assuming this time that $M=4$. Four small balls of radius $\varepsilon/M$ are located inside $B(c,\varepsilon)$ and some of them are crossed by the set $\cS$, which here is the union of three spheres (dotted lines). At least one of the $B(c_i,\varepsilon/M)$'s avoids $\cS$, for instance the red one. This is the right place to create a trap of size $M+1$ which will contain $v_M(X_n)$.}
\end{center}
\end{figure}

Thus, let us introduce the event
\[
\cH = \big\{ |\chi \cap (B(X_n, \rho) \setminus E_n)| = |\chi \cap B(c_{i_0},\varepsilon/(3M))| = M+1 \big\}
\]
for which there are exactly $M+1$ Poisson points in $B(X_n, \rho) \setminus E_n$ which are all located within the ball $B(c_{i_0},\varepsilon/(3M))$. Let us check that, on the event $\cH$, the $M$-th nearest neighbor $v_M (X_n)$ belongs to a trap of size $M+1$. On the one hand, since $B(c_{i_0},\varepsilon/M)$ does not overlap $\cS$ and $\ell \leq M-1$, $v_M(X_n)$ is definitely one of the $M+1$ Poisson points of $B(c_{i_0},\varepsilon/(3M))$. On the other hand, the annulus $B(c_{i_0},\varepsilon/M) \setminus B(c_{i_0},\varepsilon/(3M))$ being empty of Poisson points, the set $\chi \cap B(c_{i_0},\varepsilon/(3M))$ forms a trap of size $M+1$. We can now conclude. On the event $\{ X_n \textrm{ is $(\varepsilon, A)$-pioneer}\}$,
\begin{eqnarray*}
\lefteqn{\P [ v_M (X_n) \textrm{ belongs to a trap of size }M+1 \,|\, \cG_n ]}\\
& \hspace*{4cm} & \geq \; \P [ \cH \,|\, \cG_n ] \\
& & = \; \P [ |\chi \cap B(c_{i_0},\varepsilon/(3M))| = M+1 \,|\, \cG_n ] \, \P [ \chi \cap \cR = \emptyset \,|\, \cG_n ] ~,
\end{eqnarray*}
where $\cR = B(X_n,\rho) \setminus (E_n \cup B(c_{i_0},\varepsilon/(3M)) )$, using independence properties of the Poisson point process $\chi$. Since $\rho \leq A$, we can write:
\begin{equation}
\label{expression-q}
\P [ v_M (X_n) \textrm{ belongs to a trap of size }M+1 \,|\, \cG_n ] \, \geq \, \frac{\big( v_d (\varepsilon /(3M))^{d} \big)^{M+1}}{(M+1)!} e^{-v_d (\varepsilon /(3M))^{d}} \times e^{-v_d A^d}  
\end{equation}
where $v_d$ denotes the volume of the unit $d$-dimensional ball. Finally, taking $q = q(\varepsilon, A, d, M) >0$ equal to the r.h.s. of (\ref{expression-q}) achieves the proof.
\end{proof}

We are now ready to prove Lemma \ref{l:trap2}.

\begin{proof}[Proof of Lemma \ref{l:trap2}]
Let $\varepsilon, A >0$. Recall that $\V_p (\varepsilon, A)$ denotes the set of $(\varepsilon,A)$-pioneer points visited by the $K$-th nearest neighbor random walk. For any Poisson point $X$, let us define $\cA(X) = \{v_M(X) \textrm{ belongs to a trap of size }M+1\}$. On the event $\cA(X_n) \cap \{K_{n+1}=M\}$, the random walk will be locked just after step $n$ in a trap of size $M+1$. So it will not be able to visit more than $M+1$ new vertices beyond step $n$, meaning that it will not be able to admit more than $M+1$ pioneer points beyond step $n$ (since pioneer points are necessarily different vertices). Henceforth, for any integer $k$, the probability $\P[|\V_p(\varepsilon, A)| \geq k+(M+2)]$ is upperbounded by
\begin{eqnarray*}
\lefteqn{\P \Big[ \bigcup_{n_1 < \ldots < n_k} \Big(  \big\{ X_{n_1},\ldots,X_{n_k} \mbox{ are the first $k$ pioneer points}\big\} \bigcap \big( \cA(X_{n_k})^c \cup \{K_{n_{k}+1} \not= M\} \big)\Big)  \Big]} \\
& & \leq \sum_{n_1 < \ldots < n_k} \P \Big[ \big\{ X_{n_1},\ldots,X_{n_k} \mbox{ are the first $k$ pioneer points} \big\} \bigcap \big( \cA(X_{n_k})^c \cup \{K_{n_{k}+1} \not= M\} \big) \Big] \\
& & = \sum_{n_1 < \ldots < n_k} \E \Big[ \1_{ \{ X_{n_1},\ldots,X_{n_k} \mbox{ are the first $k$ pioneer points} \} } \, \P \big[ \cA(X_{n_k})^c \cup \{K_{n_{k}+1} \not= M\} \,\big|\, \cG_{n_k} \big] \Big] \\
\end{eqnarray*}
since the event $\{ X_{n_1},\ldots,X_{n_k} \mbox{ are the first $k$ pioneer points}\}$ is $\cG_{n_k}$-measurable. The random label $K_{n_{k}+1}$ being independent from the event $\cA(X_{n_k})$ and the $\sigma$-algebra $\cG_{n_k}$, we then can write:
\begin{eqnarray*}
\P \big[ \cA(X_{n_k})^c \cup \{K_{n_{k}+1} \not= M\} \,\big|\, \cG_{n_k} \big] & = & 1 - \P \big[ \cA(X_{n_k}) \,\big|\, \cG_{n_k} \big] \, \P \big[ K_{n_{k}+1} = M \big] \\
& \leq & 1 - q \, \mu(\{M\})
\end{eqnarray*}
where $0<q<1$ is the lower bound provided by Lemma \ref{l:trap}. Let us set $c = 1 - q \mu(\{M\}) \in (0,1)$ which depends on parameters $\varepsilon, A, d, M$ and $\mu(\{M\})$. Combining previous inequalities, we get
\[
\P \big[ | \V_p(\varepsilon, A) | \geq k+(M+2) \big] \leq c \, \P \big[ | \V_p(\varepsilon, A) | \geq k \big]
\]
from which the searched result follows.
\end{proof}

\subsection{Rescaling}
\label{s:renorm}

All that remains is to prove Lemma \ref{l:nice2}, asserting that among all the Poisson points visited by the random walk $(X_n)_{n\in \N}$, a positive fraction of them are pioneer points with high probability. This is the objective of Sections \ref{s:renorm} and \ref{s:nice}.

For that purpose, we use a renormalization argument. Let us partition the ambient space $\R^d$ with blocks. For any given $R \in \N^*$ and $u \in \Z^d$, let us define the block $\Lambda_R(u)$ as the hypercube of sidelength $R$ centered at $Ru$:
$$
\Lambda_R (u) \coloneqq Ru + \left[ -\frac{R}{2} ; \frac{R}{2}  \right[^d \,.
$$
The blocks $(\Lambda_R(u))_{u\in \Z^d}$ are pairwise disjoint and cover the whole space $\R^d$. We now introduce a notion of {\it good block} with the following goal in mind: whenever the random walk $(X_n)_{n\geq 0}$ gets close enough to a good block for the first time, it should go through a pioneer point. See Lemma \ref{l:nice1} of Section \ref{s:nice}.

\begin{definition}
\label{d:good}
For any given $\varepsilon >0$, $R\in \N^*$ and $u \in \Z^d$, the block $\Lambda_R(u)$ is said $\varepsilon$-good if both following statements are satisfied.
\begin{itemize}
\item[$(i)$] For any $v \in \Z^d$ s.t. $\| u-v \|_\infty \leq d$, we have $| \chi \cap \Lambda_R (v) | \geq M+1$.
%\item[$(ii)$] $\forall x \in \chi$, if $B(x, r_M(x) ) \cap \Lambda_R (u) \neq \emptyset$, then $r_M(x) \leq A$;
\item[$(ii)$] For any $x \in \chi$, $B(x, r_M(x) ) \cap \Lambda_R (u) \neq \emptyset$ implies that $B(x, r_M(x) ) \cap \Lambda_R (u)$ contains an Euclidean ball of radius $2 \varepsilon$.
\end{itemize}
\end{definition}

Let us define the family of Bernoulli random variables $( \cY_u^{(\varepsilon, R)} )_{u\in \Z^d}$ by
\begin{equation}
\label{e:defY}
\forall u\in \Z^d , \; \cY_u^{(\varepsilon, R)} = \cY_u \coloneqq \1_{\Lambda_R(u) \textrm{ is $\varepsilon$-good}} \,.
\end{equation}
When it does not matter, the upperscript $(\varepsilon, R)$ will be removed from the notation $\cY_u^{(\varepsilon, R)}$. The Poisson point process $\chi$ being translation invariant in distribution, those random variables are identically distributed, but dependent. In the rest of this section, we will prove that $(a)$ each $\cY_u^{(\varepsilon, R)}$ equals to $1$-- i.e. each block $\Lambda_R(u)$ is $\varepsilon$-good --with high probability for well chosen parameters $\varepsilon$ and $R$ (Lemma \ref{l:Y}) and that $(b)$ the $\cY_u$'s are $\kappa_6$-dependent (Lemma \ref{l:depfini}). Both results will allow us to state that the (dependent) random field $(\cY_u)_{u\in \Z^d}$ stochastically dominates a family of i.i.d. variables with Bernoulli distribution (Lemma \ref{l:LSS}).

In order to study how the $\cY_u$'s depend on each other, we require the intermediate result Lemma \ref{p:local} below. Let us first introduce a few notations. For any $u \in \Z^d$, let us define the event $\cD_R(u)$ as
\begin{equation}
\label{e:cDu}
\cD_R (u) \coloneqq \bigcap_{v \in \Z^d \,:\, \| u-v \|_\infty \leq d} \big\{ | \chi \cap \Lambda_R (v) | \geq M+1 \big\} \,.
\end{equation}
Notice that $\cD_R(u)$ depends only on $\chi \cap Q_R(u)$ where 
\begin{equation}
\label{e:Du}
Q_R (u) \coloneqq  \bigcup_{v \in \Z^d \,:\, \| u-v \|_\infty \leq d} \Lambda_R (v) \, = \, Ru + \left[  -R \left( d+ \frac{1}{2} \right) ; R \left( d+ \frac{1}{2} \right)\right[^d \,.
\end{equation}

\begin{lemma}
\label{p:local}
For any $u \in \Z^d$, $R\in \N^*$ and $y \in \chi$, the following inclusion a.s. holds:
\[
\cD_R(u) \cap \big\{ B(y, r_M(y) ) \cap \Lambda_R (u) \neq \emptyset \big\} \, \subset \, \big\{ y \in Q_R(u) \big\} \cap \big\{ r_M (y) \leq \sqrt{d} R \big\} \,.
\]
\end{lemma}

\begin{proof}
Let $u \in \Z^d$, $R\in \N^*$ and $y \in \chi$. Let us assume that the event $\cD_R(u)$ holds and $B(y, r_M(y)) \cap \Lambda_R (u) \neq \emptyset$. In a first case, let us consider that $y \in Q_R(u)$. So there exists $v \in \Z^d$ with $\| u-v \|_\infty \leq d$ such that $y \in \Lambda_R(v)$. The hypercube $\Lambda_R(v)$ having a diameter equal to $\sqrt{d} R$, it is included in the ball $B(y, \sqrt{d} R)$. Thanks to $\cD_R(u)$, the block $\Lambda_R(v)$ contains at least $M+1$ Poisson points. So the same holds for $B(y, \sqrt{d} R)$ meaning that $r_M(y)$ is at most $\sqrt{d} R$, which proves to the searched inclusion.

Now let us assume that $y \notin Q_R(u)$. We are going to prove that this assumption leads to a contradiction. Let $z \in \partial \Lambda_R (u)$ be the element of $\R^d$ minimizing the Euclidean distance between $y$ and the compact set $\partial \Lambda_R (u)$, {\it i.e.}, $d_2 (y, \partial \Lambda_R (u)) = \| y-z \|$. Thus, let $y'$ be the intersection point between the segment $[y,z]$ and $\partial Q_R (u)$, see Figure \ref{fig:inclusion}. By construction of $Q_R(u)$, $\| y'-z \|$ is larger than $dR$. So, as previously, the ball $B(y', \| y'-z\|)$ has to strictly contain some block $\Lambda_R(v)$ (with diameter $\sqrt{d} R$), with $\| u-v \|_\infty \leq d$, which itself contains at least $M+1$ Poisson points thanks to the event $\cD_R(u)$. On the other hand,
\[
B(y', \| y'-z\|) \subset B(y, \| y-z\|) \subset B(y, r_M(y))
\]
since $B(y, r_M(y)) \cap \Lambda_R (u) \neq \emptyset$ and by construction of $z$. We then have proved that the ball $B(y, r_M(y))$ includes in its interior at least $M+1$ Poisson points, contradicting the definition of $r_M(y)$.
\end{proof}

\begin{figure}
\begin{center}
\includegraphics[width=8cm,height=6cm]{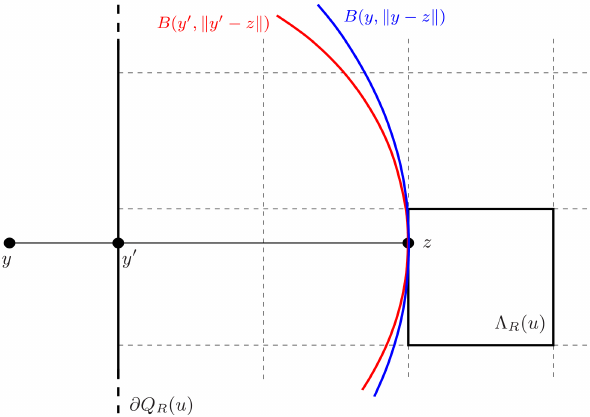}
\caption{\label{fig:inclusion} This picture represents the case where $y \notin Q_R(u)$. The balls $B(y',\|y'-z\|)$ and $B(y,\|y'-z\|)$ are resp. depicted in red and blue to illustrate the inclusion of the first one into the second one.}
\end{center}
\end{figure}

We are now well equipped to study the dependence between the random indicators $(\cY_u)_{u\in \Z^d}$.

\begin{lemma}
\label{l:depfini}
There exists a positive constant $\kappa_6 = \kappa_6 (d)$ such that, for any $\varepsilon > 0$ and $R\in \N^*$, $(\cY_u)_{u\in \Z^d}$ is a $\kappa_6$-dependent random field in the following sense: for any $U,V \subset \Z^d$ satisfying that for any $(u,v) \in U \times V$ we have $\| u - v \|_\infty > \kappa_6$, then the families $(\cY_u)_{u\in U} $ and $(\cY_v)_{v\in V}$ are independent. 
\end{lemma}

\begin{proof}
Let $u\in\Z^d$. Let us check that both items in Definition \ref{d:good} concerning the $\varepsilon$-goodness of the block $\Lambda_R (u)$ only depend on $\chi$ through a neighborhood of $Ru$. About Item $(i)$, recall that the event $\cD_R(u)$ defined in \eqref{e:cDu} depends only on $\chi \cap Q_R(u)$, where $Q_R(u)$ is defined in \eqref{e:Du}.

Lemma \ref{p:local} says that, on the event $\cD_R(u)$, a Poisson point $x$ such that $B(x, r_M(x)) \cap \Lambda_R (u) \neq \emptyset$ has to satisfy $x \in D_R(u)$ and $r_M(x) \leq \sqrt{d} R$. Hence, the ball $B(x,r_M(x))$ is included in the set
$$
Q'_R(u) \coloneqq \left\{ z_1 + z_2 \,:\, z_1 \in Q_R (u) \,,\, \|z_2\| \leq \sqrt{d} R \right\} \subset Ru + \left[ -R \left( 2d+ \frac{1}{2} \right) ; R \left( 2d+ \frac{1}{2} \right)\right[^d
$$
and the event occurring in Item $(ii)$ of Definition \ref{d:good}, namely
$$
\big\{ B(x, r_M(x) ) \cap \Lambda_R (u) \textrm{  contains an Euclidean ball of radius }2 \varepsilon \big\} \,,
$$
only depends on $\chi \cap D'_R(u)$. Therefore $\cY_u$ only depends on $\chi \cap D'_R(u)$ which proves that $(\cY_u)_{u\in \Z^d}$ is $\kappa_6$-dependent with
$$
\kappa_6 = \kappa_6 (d) \coloneqq 2 \Big( 2d + \frac{1}{2} \Big)\,.
$$
\end{proof}

It is possible to pick $R$ large enough and $\varepsilon > 0$ small enough so that the probability $\P[ \cY_0^{(\varepsilon, R)} = 1 ]$ is as large as we want.

\begin{lemma}
\label{l:Y}
For any $\eta' \in (0,1)$, there exist $R_1 = R_1 (\eta ', d, M)\in \N^*$ and $\varepsilon_1 = \varepsilon_1 (\eta ', d, M) >0$ such that
\[
\P \big[ \cY_0^{(\varepsilon_1, R_1)} = 1 \big] \geq \eta' ~.
\]
\end{lemma}

\begin{proof}
Let $\eta' \in (0,1)$. Recall that the event $\cD_R (0)$ is defined in \eqref{e:cDu} by
$$
\cD_R (0) = \bigcap_{v \in \Z^d \,:\, \| v \|_\infty \leq d} \big\{ | \chi \cap \Lambda_R (v) | \geq M+1 \big\} \,.
$$
The properties of the Poisson point process $\chi$ allow us to compute
\begin{eqnarray*}
\P \big[ \cD_R(0) \big] & = & \P \big[ \big\{ | \{ \chi \cap \Lambda_R (0) \} | \geq M+1 \big\} \big] ^{(2d + 1)^d} \\
& = & \big( 1 - \P [\cP^{R^d} \leq M] \big)^{(2d + 1)^d}
\end{eqnarray*}
where $\cP^{R^d}$ is a random variable with Poisson distribution with parameter $\cL^d (\Lambda_R (0)) = R^d$. Thus, using
$$
\P \big[ \cP^{R^d} \leq M \big] = \sum_{k=0}^M \frac{R^{dk}}{k!} e^{-R^d} \leq (M+1) R^{dM} e^{-R^d}
$$
which tends to $0$ as $R\to +\infty$, we can therefore pick $R_1 = R_1 (\eta', d, M) \in \N^*$ large enough so that
$$
\P \big[ \cD_{R_1}(0) \big] = \Big( 1 - (M+1) R_1^{dM} e^{-R_1^d} \Big)^{(2 d + 1)^d}  \geq \sqrt{\eta'} \,.
$$

Given this value $R_1$, we now tune the parameter $\varepsilon$. Let $x \in \chi$ such that $B(x, r_M(x)) \cap \Lambda_{R_1} (0) \neq \emptyset$. A.s. the ball $B(x, r_M(x))$ is not tangent to the block $\Lambda_{R_1}(0)$. This means that the random variable
$$
\cR (x) \coloneqq \sup \big\{ r\in \Q \,:\, \exists c \in \Q^d \,,\, B (c, r) \subset B(x, r_M(x) ) \cap \Lambda_{R_1} (0) \big\}
$$
is a.s. positive. The rational sets $\Q$ and $\Q^d$ are used in the definition above to ensure the measurability of $\cR(x)$. Thus, let us consider the (non-negative) random variable
$$
\cV \coloneqq \inf \big\{ \cR (x) \,:\, x \in \chi \; \mbox{ and } \; B(x, r_M(x) ) \cap \Lambda_{R_1} (0) \neq \emptyset \big\} ~.
$$
On the event $\cD_{R_1}(0)$, any $x \in \chi$ such that $B(x, r_M(x) ) \cap \Lambda_{R_1} (0) \neq \emptyset$ has to belong to the set $Q_{R_1}(0)$ thanks to Lemma \ref{p:local}. So, on the event $\cD_{R_1}(0)$, the infimum $\cV$ can be expressed as
\begin{equation}
\label{cV-min}
\cV \coloneqq \min \big\{ \cR (x) \,:\, x \in \chi \cap Q_{R_1}(0) \; \mbox{ and } \; B(x, r_M(x) ) \cap \Lambda_{R_1} (0) \neq \emptyset \big\}
\end{equation}
since the set $\chi \cap Q_{R_1} (0)$ is a.s. finite. Because each $\cR(x)$ involved in (\ref{cV-min}) is positive, their minimum is a.s. positive too. Therefore, we can choose $\varepsilon_1 = \varepsilon_1 (\eta', d, M) > 0$ small enough such that
$$
\P \big[ \cV \geq  2\varepsilon_1 \,|\, \cD_{R_1}(0) \big] \geq \sqrt{\eta'} \,.
$$
Notice that the conditional distribution of $\cV$  depends  only on parameters $\eta'$, $d$ and $M$. Let us conclude:
\[
\P \big[ \cY_0^{(\varepsilon_1, R_1)} = 1 \big] \geq \P \big[ \cD_{R_1} (0) \cap \{ \cV \geq 2 \varepsilon_1 \} \big] = \P \big[ \cV \geq  2\varepsilon_1 \,|\, \cD_{R_1}(0) \big] \, \P \big[ \cD_{R_1}(0) \big] \geq \eta' ~.
\]
\end{proof}

This section ends with the announced stochastic domination result.

\begin{lemma}
\label{l:LSS}
For any $\eta \in (0,1)$, there exist $\varepsilon_0>0$ and $R_0\in \N^*$ (only depending on parameters $\eta,d,M$) such that the random field $(\cY_u^{(\varepsilon_0, R_0)} )_{u\in \Z^d}$ stochastically dominates a family $(\cZ_u)_{u\in \Z^d}$ of i.i.d. Bernoulli random variables with parameter $\eta$.
\end{lemma}

\begin{proof}
This is an application of the well known stochastic domination result by Liggett, Schonmann and Stacey (see Theorem 0 in \cite{LSS}). Let $\eta \in (0,1)$ and consider a family $(\cZ_u)_{u\in \Z^d}$ of i.i.d. Bernoulli random variables with parameter $\eta$. First, Lemma \ref{l:depfini} asserts that $(\cY_u^{(\varepsilon, R)} )_{u\in \Z^d}$ is a $\kappa_6$-dependent random field with $\kappa_6 = \kappa_6(d)$. Then, Theorem 0 in \cite{LSS} says that there exists some constant $\eta' \in (0,1)$, depending on $\eta,d$ and the range of dependence $\kappa_6 (d)$ but uniform on $\varepsilon,R$, such that if there exists a couple of parameters $(\varepsilon, R)$ satisfying $\P [ \cY_0^{(\varepsilon, R)}  =1 ] \geq \eta'$, then $(\cY_u^{(\varepsilon, R)} )_{u\in \Z^d} $ dominates stochastically $(\cZ_u)_{u\in \Z^d}$. The latter condition is ensured by Lemma \ref{l:Y}.
\end{proof}

\subsection{Good blocks and pioneer points}
\label{s:nice}

Let us now introduce the notion of {\it nice points}. Roughly speaking, the $n$-th step $X_n$ of the $K$-th nearest neighbor random walk is a nice point if $X_n$ is close to a good block that has not been explored by the process yet.

\begin{definition}
\label{d:nice}
For any $\varepsilon >0$, $R\in \N^*$ and $n\in \N$, $X_n$ is called an $(\varepsilon,R)$-nice point if there exists $u \in \Z^d$ such that
\begin{itemize}
\item[$(i)$] $\Lambda_R (u)$ is $\varepsilon$-good;
\item[$(ii)$] $E_n \cap \Lambda_R(u) = \emptyset$;
\item[$(iii)$] $E_{n+1} \cap \Lambda_R(u) \neq  \emptyset$.
\end{itemize}
\end{definition}

Nice points are useful for two reasons: nice points are pioneer points (Lemma \ref{l:nice1} below) and they occur along the trajectory of $(X_n)_{n\geq 0}$ proportionally with the number of good blocks that the random walk has encountered (Lemma \ref{l:pionnier-boite} further).

\begin{lemma}
\label{l:nice1}
For any $\varepsilon>0, R\in \N^*$ and $ n\in \N$, if $X_n$ is an $(\varepsilon, R)$-nice point then
\begin{itemize}
\item[$(i)$] $r_M(X_n)\le \sqrt{d}R$;
\item[$(ii)$] $X_n$ is a $(\varepsilon, \sqrt{d} R)$-pioneer point.
\end{itemize}
\end{lemma}

\begin{proof}
Let $\varepsilon > 0$, $R \in \N^*$ and $n \in \N$. Let us assume that $X_n$ is an $(\varepsilon, R)$-nice point and let us denote by $x$ the Poisson point such that $X_n = x$. Let $u \in \Z^d$ be the vertex satisfying Items $(i)$-$(iii)$ of Definition \ref{d:nice}. Since $E_{n+1} \setminus E_{n} \subset B(x, r_M (x))$, the fact that $X_n = x$ is an $(\varepsilon, R)$-nice point implies that the ball $B(x, r_M (x))$ overlaps the good block $\Lambda_R (u)$. This has two consequences. First, by Lemma \ref{p:local}, the radius $r_M(x)$ has to be smaller than $A := \sqrt{d} R$, which establishes Item (i) of the lemma.
Second, $B(x,r_M(x)) \cap \Lambda_R (u)$ contains an Euclidean ball with radius $2\varepsilon$, say $B(c,2\varepsilon)$. 

We can then prove that $X_n = x$ is an $(\varepsilon, A)$-pioneer point. Let us set $\rho \coloneqq r_M(x) - \varepsilon$ which is positive since $r_M(x) \geq 4 \varepsilon$. Indeed we know that $B(c,2\varepsilon)\subset B(x,r_M(x)) \setminus \{x\}$. We can immediately write $0 < \rho < r_M (x) \leq A$. Moreover, the ball $B(c,\varepsilon)$ is included in $B(x,\rho)$ by choice of $\rho$ and also in the block $\Lambda_R(u)$ which itself does not intersect $E_n$ (Item $(ii)$ of Definition \ref{d:nice}). So,
$$
B(c,\varepsilon) \subset B(x,\rho) \cap (\R^d \setminus E_n) \,,
$$
proving Item $(i)$ of Definition \ref{d:pioneer}. Finally, $\rho < r_M(x)$ implies that
$$
\big| \big\{ (\chi \setminus \{x\}) \cap (B (x,\rho) \cap E_n) \big\} \big| \leq \big| \big\{ (\chi \setminus \{x\}) \cap B (x,\rho) \big\} \big| < M  \,,
$$
proving Item $(ii)$ of Definition \ref{d:pioneer}.
\end{proof}

For any $n\in \N$, we set
$$
\fE_n \coloneqq \{ u\in \Z^d \,:\, \Lambda_R(u) \cap E_n \neq \emptyset \}
$$
and $\fE_{\infty} = \cup_{n\geq 0} \fE_n$. Let $\cA_m$ be the set of all the animals with size $m$ in $\Z^d$-- an animal of size $m$ is a connected subset of $\Z^d$, with cardinality $m$, containing the origin $0$. Let also $\cA \coloneqq \cup_{m\geq 1} \cA_m$ be the set of all animals in $\Z^d$. Notice that, for any $n\in \N$, $\fE_n \in \cA$ and that $(\fE_n)_{n\in \N}$ is non decreasing w.r.t. $n$ for the inclusion.

The following result compares the number of pioneer points and the number of good blocks in the vicinity of the trajectory $(X_n)_{n \in \N}$.

\begin{lemma}
\label{l:pionnier-boite}
There exists a constant $\kappa_7 = \kappa_7 (d) > 0$ such that, for any $\varepsilon > 0$ and $R \in \N^*$, the following inequality holds:
$$
| \V_p (\varepsilon, \sqrt{d} R) | \geq  \kappa_7 \, \big| \big\{ u \in \fE_{\infty} \,:\,  \cY^{(\varepsilon, R)}_u = 1 \big\} \big|  \,,
$$
where $\V_p (\varepsilon, \sqrt{d} R) = \big\{ x \in \chi \,:\, \exists n \geq 0 , \, X_n = x \mbox{ and $X_n$ is $(\varepsilon, \sqrt{d} R)$-pioneer} \big\}$.
\end{lemma}

\begin{proof}
Let $\varepsilon > 0$ and $R \in \N^*$. Given $u \in \fE_{\infty}$, we say that the block $\Lambda_R(u)$ is discovered by $X_n$ when $n$ is the minimal integer $m$ for which $u \in \fE_m$. Hence,
\[
\big| \big\{ u \in \fE_{\infty} \,:\, \cY^{(\varepsilon, R)}_u =1 \big\} \big| = \sum_{n \geq 0} \big| \big\{ u \in \Z^d \,:\, \mbox{ $\Lambda_R(u)$ is $\varepsilon$-good and discovered by $X_n$} \big\} \big| ~.
\]
When $X_n$ discovers the $\varepsilon$-good block $\Lambda_R(u)$, it is necessarily a $(\varepsilon, R)$-nice point whose radius $r_M(X_n)$ is smaller than $\sqrt{d} R$ by Lemma \ref{l:nice1}. So,
\begin{equation}
\label{borneNice}
\big| \big\{ u \in \fE_{\infty} \,:\, \cY^{(\varepsilon, R)}_u =1 \big\} \big| \leq \sum_{n \geq 0} \big| \big\{ u \in \Z^d \,:\, B(X_n, \sqrt{d} R) \cap \Lambda_R(u) \not= \emptyset \big\} \big| \, \1_{\mbox{$X_n$ is $(\varepsilon, R)$-nice}} ~.
\end{equation}
Assuming for the moment that, for any integer $n$,
\begin{equation}
\label{borneDiscover}
\big| \big\{ u \in \Z^d \,:\, B(X_n, \sqrt{d} R) \cap \Lambda_R(u) \not= \emptyset \big\} \big| \leq ( 3 \sqrt{d} + 1 )^d ~,
\end{equation}
we can conclude thanks to (\ref{borneNice}), (\ref{borneDiscover}) and Lemma \ref{l:nice1}. Indeed,
\begin{eqnarray*}
\big| \big\{ u \in \fE_{\infty} \,:\, \cY^{(\varepsilon, R)}_u =1 \big\} \big| & \leq & ( 3 \sqrt{d} + 1 )^d \sum_{n \geq 0} \1_{\mbox{$X_n$ is $(\varepsilon, R)$-nice}} \\
& \leq & ( 3 \sqrt{d} + 1 )^d \sum_{n \geq 0} \1_{\mbox{$X_n$ is $(\varepsilon, \sqrt{d} R)$-pioneer}} \\
& \leq & ( 3 \sqrt{d} + 1 )^d | \V_p (\varepsilon, \sqrt{d} R) |
\end{eqnarray*}
since two pioneer points have to be different Poisson points. The constant $\kappa_7 (d) := ( 3 \sqrt{d} + 1 )^{-d}$ is suitable.

It then remains to prove (\ref{borneDiscover}). Since the diameter of the hypercube $\Lambda_R(u)$ is equal to $\sqrt{d} R$, the constraint $B(X_n, \sqrt{d} R) \cap \Lambda_R(u) \not= \emptyset$ forces the distance $\|X_n - Ru\|$ to be at most $\frac{3}{2}\sqrt{d} R$. Henceforth,
\begin{eqnarray*}
\big| \big\{ u \in \Z^d \,:\, B(X_n, \sqrt{d} R) \cap \Lambda_R(u) \not= \emptyset \big\} \big| & \leq & \Big| \Big\{ u \in \Z^d \,:\, \Big\| \frac{X_n}{R} - u \Big\| \leq \frac{3}{2}\sqrt{d} \Big\} \Big| \\
& \leq & \Big| \, \Z^d \cap \Big( \frac{X_n}{R} + \Big[ -\frac{3}{2}\sqrt{d} , \frac{3}{2}\sqrt{d} \Big]^d \Big) \Big| \\
& \leq & ( 3 \sqrt{d} + 1 )^d ~.
\end{eqnarray*}
\end{proof}

We can now complete the proof of Lemma \ref{l:nice2}.

\begin{proof}[Proof of Lemma \ref{l:nice2}]
For any $\varepsilon>0, A>0, C>0, D>0, k\in \N^*$, we have
\begin{align}
\label{e:nice1}
 \P [ | \V | \geq Dk \,,\, | \V_p (\varepsilon, A) |  < C k ] & \leq \P \left[ | \V_p (\varepsilon, A) | < C k \,,\, |\fE_{\infty} | \geq  k \right] \nonumber \\ & \qquad +  \P \left[  |\fE_{\infty} | <  k  \,,\, \left\vert \chi \cap \left( \bigcup_{u \in \fE_{\infty}} \Lambda_R(u) \right) \right\vert \geq D k \right]\,.
\end{align}
We control the two terms on the right hand side of Equation \eqref{e:nice1} separately. Let us begin with the first term. Before entering into the technical details, we briefly outline the strategy. This term corresponds to the case where the number of blocks close to the trajectory of the random walk is large, while the number of pioneer points along the trajectory remains small. However, near each good block visited by the walk, a nice point (hence a pioneer point) appears along the trajectory. Therefore, if the trajectory passes close to many blocks but produces only few pioneer points, it must encounter many bad blocks ({\it i.e.}, blocks that are not good). Such a situation is highly unlikely when the probability that a block is good is sufficiently close to one. We now turn to the details.

For a given $R\in \N^*$, define $A \coloneqq \sqrt{d} R$. By Lemma \ref{l:pionnier-boite}, we know that
$$ | \V_p (\varepsilon, \sqrt{d} R) | \geq \kappa_7 | \{ u\in  \fE_{\infty} \,:\,  \cY^{(\varepsilon, R)}_u = 1 \}| \,,$$
thus
\begin{align*}
\P \left[ | \V_p (\varepsilon, \sqrt{d} R ) | < C k \,,\, |\fE_{\infty} | \geq  k \right] & \leq \P \left[  |\fE_{\infty} | \geq  k \,,\, \sum_{u\in \fE_{\infty}} \cY_u^{(\varepsilon, R)} \leq C \kappa_7^{-1} k\right] \\ 
& \leq  \P \left[ \exists \fA \in \cA_k \,:\, \sum_{u\in \fA } \cY_u^{(\varepsilon, R)} \leq C \kappa_7^{-1} k  \right]\\
& \leq \sum_{\fA \in \cA_k} \P \left[  \sum_{u\in \fA } \cY_u^{(\varepsilon, R)} \leq C \kappa_7^{-1} k  \right]
\,.
\end{align*}
By Equation (4.24) in \cite{Grimmett}, we know that for any $k \in \N^*$,
\begin{equation}
\label{e:nice3}
|\cA_k| \leq 7^{dk} \,.
\end{equation}
For any given $\eta\in(0,1)$, let us pick $R_0 = R_0 (\eta, d, M) \in \N^*$ and $\varepsilon_0 = \varepsilon_0 (\eta, d, M) > 0$ as in Lemma \ref{l:LSS} such that $(\cY_u^{(\varepsilon_0, R_0)} )_{u\in \Z^d}$ stochastically dominates a family $(\cZ_u)_{u\in \Z^d}$ of i.i.d. Bernoulli random variables with parameter $\eta$. Therefore
\begin{align*}
\P \left[ | \V_p (\varepsilon_0, \sqrt{d} R_0 ) | < C k \,,\, |\fE_{\infty} | \geq  k \right] & \leq \sum_{\fA \in \cA_k} \, \P \left[  \sum_{u\in \fA } \cY_u^{(\varepsilon_0, R_0)} \leq C \kappa_7^{-1} k  \right] \\
& \leq \sum_{\fA \in \cA_k} \, \P \left[  \sum_{u\in \fA } \cZ_u \leq C \kappa_7^{-1} k  \right]\\
& \leq 7^{dk} \, \P \left[ \cU^{k, \eta} \leq C \kappa_7^{-1} k \right] \,,
\end{align*}
where $\cU^{k, \eta} $ is a random variable with a binomial distribution with parameters $(k, \eta)$. Fix $C \coloneqq \kappa_7 /2$, thus $C \kappa_7^{-1} k = k /2$. Then there exists a constant\footnote{Indeed for any $k\in \N^*$ we have $ \log \P \left[ \cU^{k, \eta} \leq   k /2  \right] = \log \P \left[ \cU^{k, 1-\eta} \geq   k /2  \right]\leq -k  K^*_{1-\eta} (1/2) $ where for any $x, p\in [0,1]$ we have $K^*_{p} (x) = \sup\{tx - \log (p e^t + (1-p)) \,:\, t>0\} = x \log (x/p) + (1-x) \log ((1-x) / (1-p))$ the Cram\'er function associated with the Bernoulli distribution with parameter $p$.} $c_1 (\eta) \in (0,\infty)$ such that for any $k\in \N^*$,
\begin{equation}
\label{e:nice5}
 \P \left[ \cU^{k, \eta} \leq   k /2  \right] \leq  e^{- c_1 (\eta) k} \,,
 \end{equation}
where $c_1 (\eta) $ satisfies $\lim_{\eta \rightarrow 1} c_1  (\eta) = +\infty $. 
We fix $\eta_0 = \eta_0 (d) $ such that $e^{- c_1 (\eta_0)} \leq 7^{-2 d}$. Therefore, for this $\eta_0 (d)$ and the corresponding $R_0 (\eta_0,d,M)$ and $\varepsilon_0 (\eta_0,d,M)$ given by Lemma \ref{l:LSS}, for the constant $\kappa_7 = \kappa_7 (d)$ given by Lemma \ref{l:pionnier-boite}, using Equation \eqref{e:nice3}, we obtain that for any $ k\in \N^*$,
\begin{equation}
\label{e:nice2}
\P \left[ | \V_p (\varepsilon_0, \sqrt{d} R_0 ) | < \frac{\kappa_7 }{2} k \,,\, |\fE_{\infty} | \geq  k \right]  \leq   7^{dk} \,e^{- c_1 (\eta_0) k} \leq   7^{-dk} \,.
\end{equation}

We now study the second term on the right hand side of Equation \eqref{e:nice1}. The idea is as follows. This term corresponds to the case where the trajectory of the random walk gets close to only a small number of blocks, while visiting a large number of distinct points of $\chi$. However, the number of points of $\chi$ inside any given collection of blocks follows a Poisson distribution, and is therefore very unlikely to take large values. It is therefore enough to use a union bound to take into account all the possible locations for $\fE_{\infty}$. We now give the details.

By Equation \eqref{e:nice3} we have for any $D\in \N^*, k\in \N^*$, for $R_0 = R_0 (\eta_0,d,M)$ as fixed previously,
\begin{align*}
\P \left[  |\fE_{\infty} | < k  \,,\, \left\vert \chi \cap \left( \bigcup_{u \in \fE_{\infty}} \Lambda_{R_0}(u) \right) \right\vert \geq D k \right] & \leq \sum_{\fA \in \cA_k} \P \left[  \left\vert \chi \cap \left( \bigcup_{u \in \fA }\Lambda_{R_0}(u) \right) \right\vert \geq D k  \right] \\
& \leq 7^{dk} \, \P \left[  \cP^{kR_0^d} \geq D k  \right] \,,
\end{align*}
where $\cP^{kR_0^d} $ is a random variable with Poisson distribution with parameter $kR_0^d$. There exists a constant\footnote{Indeed for any $k\in \N^*$ we have $\log \P \left[ \cP^{kR_0^d} \leq  Dk \right] \leq - k K^*_{R_0^d} ( D) $ where $K_{\ell }^*$ is the Cram\'er function associated with the Poisson distribution with parameter $\ell$, {\it i.e.}, $K^*_{\ell} (x)= \sup\{tx -  \ell (e^t - 1) \,:\, t>0\} = \ell - x + x \log (x/\ell)$ for any $x>0$.} $c_2 (R_0,D) \in (0,\infty)$ such that for any $k\in \N^*$,
\begin{equation}
\label{e:nice6}
 \P \left[ \cP^{kR_0^d} \geq D k\right] \leq   e^{- c_2 (R_0,D) k} \,,
 \end{equation}
where $c_2 (R_0,D)$ satisfies $\lim_{D \rightarrow +\infty } c_2 (R_0,D) = +\infty $.
We fix $D_0 = D_0 (R_0,d)$ such that $e^{- c_2 (R_0,D_0)} \leq 7^{-2 d}$. Therefore, for this $D_0 = D_0(R_0, d)$, for $R_0 = R_0 (d,M)$ as fixed previously, using Equation \eqref{e:nice3}, we obtain that for any $ k\in \N^*$,
\begin{equation}
\label{e:nice4}
\P \left[  |\fE_{\infty} | < k  \,,\, \left\vert \chi \cap \left( \bigcup_{u \in \fE_{\infty}} \Lambda_{R_0}(u) \right) \right\vert \geq D_0 k \right]  \leq 7^{dk} \,e^{- c_2(D_0) k} \leq 7^{-dk} \,.
\end{equation}
We conclude the proof of Lemma \ref{l:nice2} by combining Equations \eqref{e:nice1}, \eqref{e:nice2} and \eqref{e:nice4}.
\end{proof}

%%%%%%%%%%%%%%%%%%%%

\section{Length of the trajectory}
\label{s:L}

This section is devoted to the proof of Theorem \ref{t:L}, {\it i.e.}, the control of the tail distribution of the length $\L$ of the trajectory. Suppose Assumption (BS) is satisfied. 
For all $n \ge 1$, 
\begin{equation}\label{e:decoupage_initial} 
	\P[\L \ge An] \le \P[|\V| \ge n] + \P[\cE_n]
\end{equation}
where 
\[
	\cE_n = \{\L \ge An \text{ and } |\V| \le n\}
\]
and $A>0$ is a large constant depending on $d$ and $M$ which will be chosen at the end of the proof.
By Theorem \ref{t:V}, the first term is bounded above by $C_1 e^{-C_2 n}$.
We thus focus on the second term.

Assume that $\cE_n$ occurs.
First we associate an embedded rooted tree with the random walk $(X_n)_{n\geq 0}$.
Its set of vertices is the range $\V$ of the walk.
Its root is the origin.
For any $n \ge 1$, we put an edge between $X_{n-1}$ and $X_n$ when the vertex $X_n$ is visited for the first time at time $n$.
Equivalently, we can construct the tree dynamically as follows. 
At time $0$, the tree consists solely of its root $\{0\}$. 
As time evolves, whenever the walk visits a new vertex, that vertex is attached by an edge to the vertex from which it was first reached.

For our purposes, it is more convenient to separate the combinatorial structure from its spatial embedding.
More precisely, from the embedded tree constructed above we extract two objects:
\begin{enumerate}
	\item A rooted tree,
	\item An embedding $x$, that is a list of spatial positions $x(v)$ -- actually points of $\chi$ -- indexed by the vertices $v$ of the tree.
\end{enumerate}

In Lemma \ref{l:en_tree_v1} below we list the useful properties of the tree and its embedding.
To state it, we introduce the following notation.
Let $m \ge 0$. 
We denote by $\cT_m$ the set of rooted trees with $m+1$ vertices (equivalently, $m$ edges).
We say that the size of a tree $T \in \cT_m$ is $m$ and we write $\size(T)=m$.
We orient the edges of $T$ away from the root $o$. 
For vertices $v,w$ of $T$, we write $v \to w$ if $v$ is the parent of $w$ (equivalently, if $w$ is a child of $v$).
We define an embedding of $T$ as a family $x(v)$ of distinct points of $\chi$ indexed by the vertices of $T$ such that $x(o)=0$.
Recall that $o$ is the root and that we work under the Palm distribution so $0$ is indeed a point of $\chi$.

\begin{lemma}\label{l:en_tree_v1} On $\cE_n$, there exist $m \in \{1,\dots,n-1\}$, $T \in \cT_m$ and an embedding $x$ of $T$ such that the following holds.
	\begin{enumerate}
		\item For all vertices $v$ of $T$, $|B(x(v),r_M(x(v)) \cap \chi| = M+1$.
		\item For all edges $v \to w$ of $T$, $\|x(w)-x(v)\| \le r_M(x(v))$.
		\item $An \le M \sum_{v} r_M(x(v))$.
	\end{enumerate}
\end{lemma}

\begin{proof} 
	Assume that $\cE_n$ holds. We defined the tree $T$ and the embedding $x$ in the above discussion.
	Note that $T$ belongs to $\cT_m$ for some $m \in \{1,\dots,n-1\}$.
	Let us now collect the constraints. 
	
	The first item is true by definition of $r_M$.
	
	For all edges $v \to w$ of $T$, $x(w)$ has been first visited by the random walk from $x(v)$,
	and therefore $x(w)$ is the $k$-th neighbor of $x(v)$ for some $k \in \{1,\dots,M\}$. 
	This yields the second item.
	
	The length $\L$ is at least $An$.
	The couple $(T,x)$ does not determine $\L$, but it provides the following upper-bound:
	\[
		\L \le \sum_{v\in T} M \, r_M(x(v)).
	\]
	Indeed, $r_M(x(v))$ is a bound on the length of any steps made by the walk from $x(v)$ and there are at most $M$ such distinct steps.
	Recall that multiple traversals from one vertex to one other are counted only once in $\L$.
	The last item follows.
\end{proof}

Discretizing the $r_M$, we immediately get the following version in which $\N$ denotes the set of non-negative integers.

\begin{lemma}\label{l:en_tree_v2} On $\cE_n$, there exist $m \in \{1,\dots,n-1\}$, $T \in \cT_m$, an embedding $x$ of $T$ 
and a family $r(v)$ of elements of $\N$ indexed by the vertices of $T$ such that the following holds.
	\begin{enumerate}
		\item For all vertices $v$ of $T$, $r(v) \le r_M(x(v))$.
		\item For all vertices $v$ of $T$, $|B(x(v),r(v))| \cap \chi| \le M+1$.
		\item For all edges $v \to w$ of $T$, $\|x(w)-x(v)\| \le r(v)+1$.
		\item $An \le M \sum_{v} (r(v)+1)$.
	\end{enumerate}
\end{lemma}
\begin{proof}
	The lemma follows by Lemma \ref{l:en_tree_v1} choosing $r(v)=\lfloor r_M(x(v)) \rfloor$ for each vertex $v$.
\end{proof}

Fix $y_1, \dots, y_r$ distinct elements of $\R^d$.
We need to consider the point process $\chi$ both under its stationary distribution 
and under its $r$-order Palm distribution at the points $y_1, \dots, y_r$ for various choices of $y_1, \dots, y_r$.
We use the notation $\tilde\P$ for the former and $\tilde \P_{y_1,\dots,y_r}$ for the latter.
Thus, under $\tilde\P$, $\chi$ is a homogeneous Poisson point process with intensity $1$ and
\begin{equation}\label{e:palm_ajout}
	\tilde \P_{y_1,\dots,y_r} [\chi \in \cdot] = \tilde \P [\chi \cup \{y_1,\dots , y_r\} \in \cdot].
\end{equation}
As usual we denote by $\tilde \E$ and $\tilde \E_{y_1,\dots,y_r}$ the corresponding expected values.
Actually, we only need the multivariate Mecke equation (Theorem 4.4 in \cite{last_penrose_2017}) which states, in our context,
\begin{equation} \label{e:multi-mecke}
    \tilde\E\left[\sum_{y_1,\dots,y_r \in \chi}^{\neq} f(y_1,\dots,y_r,\chi) \right]
    = \int_{(\R^d)^r} \tilde\E_{y_1,\dots,y_r}  [f(y_1,\dots,y_r,\chi) ] \d y_1 \cdots \d y_r
\end{equation}
where $f$ is non-negative and measurable and where the superscrip $\neq$ in the sums means that the $y_1, \dots, y_r$ are assumed to be distinct.

Now fix $x_1, \dots, x_m$ distinct elements of $\R^d \setminus \{0\}$ and set $x_0=0$. 
We wish to work with embeddings given by $x_0, x_1, \dots, x_m$.
For this purpose, we need to label the vertices of our trees.
Therefore, for each tree $T \in \cT_m$, we fix an ordering of its vertices and denote them by $v_0, \dots, v_m$. 
The ordering is arbitrary, with the sole requirement that $v_0$ is the root.

For a tree $T \in \cT_m$,  $x_0, \dots, x_m$ as above and $r_0,\dots,r_m \in \N$, we define the event $\cF_n(T,x,r)$ 
(this makes sense under the Palm distribution $\tilde\P_{x_0,\dots,x_m}$) by
\begin{subequations}\label{Fn}
	\begin{align}
		\cF_n(T,x,r)
		=
		\big\{
			&\text{for all vertices }v_i\text{ of }T, \quad r_i \le r_M(x_i), \label{Fn:discretisation}\\
			&\text{for all vertices }v_i\text{ of }T, \quad |B(x_i,r_i) \cap \chi| \le M+1, \label{Fn:pas_trop_de_points}\\
			&\text{for all edges }v_i \to v_j\text{ of }T, \quad \|x_j-x_i\| \le r_i+1, \label{Fn:pas_loin}\\
			&An \le M \sum_{0 \le i \le \size(T)} (r_i+1)\label{Fn:tres_loin}
		\big\}.
	\end{align}
\end{subequations}

Note that Conditions \eqref{Fn:pas_loin} and \eqref{Fn:tres_loin} are actually deterministic.

\begin{lemma}\label{l:en_tree} We have
	\[
		\P[\cE_n] 
		\le 
		\sum_{m=1}^{n-1} \sum_{T \in \cT_m} \sum_{r\in \N^{m+1}} \int_{(\R^d)^m}  	\tilde \P_{0,x_1,\dots,x_m}[\cF_n(T,x,r)] \d x_1 \cdots \d x_m .
	\]
\end{lemma}

\begin{proof} 
	The lemma follows from Lemma \ref{l:en_tree_v2} by union bound and Mecke formula \eqref{e:multi-mecke}.
    In particular we use Mecke formula as follows.
    Given $m$, $T$ and $r$,
    \begin{align*}
        \E\left[\sum_{x_1,\dots , x_m \in \chi \setminus \{0\}}^{\neq} \1_{\cF_n(T,(0,x_1,\dots,x_m),r)}(\chi)\right] 
        & = \tilde \E_{0}\left[\sum_{x_1,\dots , x_m \in \chi \setminus \{0\}}^{\neq} \1_{\cF_n(T,(0,x_1,\dots,x_m),r)}(\chi)\right] \\
        & = \tilde\E\left[\sum_{x_1,\dots , x_m \in \chi}^{\neq} \1_{\cF_n(T,(0,x_1,\dots,x_m),r)}(\chi \cup \{0\})\right]  \\
        & = \int_{(\R^d)^m} \tilde\E\left[\1_{\cF_n(T,(0,x_1,\dots,x_m),r)}(\chi \cup \{0,x_1,\dots,x_m\})\right] \d x_1 \dots \d x_m \\
        & = \int_{(\R^d)^m} \tilde\P_{0,x_1,\dots,x_m}[\cF_n(T,(0,x_1,\dots,x_m),r)]  \d x_1 \dots \d x_m. 
    \end{align*}
    We used the definition of $\P$ and $\P_0$ in the first equality, \eqref{e:palm_ajout} in the second and fourth equality, and \eqref{e:multi-mecke} in
    the third equality.
\end{proof}

The conditions on the $|B(x_i,r_i) \cap \chi|$ are strongly dependent as the $B(x_i,r_i)$ overlap. The following geometric lemma, which is essentially Stone's Lemma (see \cite{Devroye-Gyorfi-Lugosi}), will enable us to manage them by ensuring that the overlap is bounded. See the left part of Figure \ref{fig:kissing} for an illustration in dimension $2$. If $S$ is an infinite and locally finite subset of $\R^d$, for any $x \in S$, we define $r_M(x,S)$ as we did when $S=\chi$: $r_M(x,S)$ is the distance from $x$ to its $M$-th nearest neighbor in $S$. 

\begin{lemma}\label{l:stone} 
There exists a constant $\kappa_8(d) > 0$ such that the following holds.
Let $S$ be an infinite and locally finite subset of $\R^d$.
\begin{enumerate}
\item For all $x_0 \in \R^d$ such that $x_0 \not \in S$,
\[
\big|\{x \in S : x_0 \in B(x,r_M(x,S))\}| \le M\kappa_8(d).
\]
\item 
Let $S_f$ be a finite subset of $S$.
Let $(r(x))_{x \in S_f}$ be a family of non-negative real numbers such that, for all $x \in S_f$, $r(x) \le r_M(x,S)$.
Then 
\[
\cL^d\left(\bigcup_{x \in S_f} B(x,r(x))\right) \ge \frac 1 {M\kappa_8(d)} \sum_{x \in S_f} \cL^d(B(x,r(x))).
\]
\end{enumerate}
\end{lemma}

\begin{proof} This is essentially the proof of Stone's Lemma. 
We refer to \cite{Devroye-Gyorfi-Lugosi}, Section 5.3, for a more detailed proof with explicit constants.
We give here a short proof of Lemma \ref{l:stone} to make the paper self-contained. Fix some angle $\theta \in (0,\pi/6)$.
To ease notations we assume $x_0=0$ (and thus $0 \not\in S$).
Consider for any unit vector $u$ the cone with opening angle $\theta$ and axis $u$ defined by
\[
\cone(u) =  \{x \in \R^d   : x \cdot u \ge \cos(\theta) \|x\|\}.
\]
As $\theta$ is small enough, 
\begin{equation}
\label{e:key_stone}
\text{for any $x,x' \in \cone(u)$, $\|x'\|>\|x\|$ implies $\|x'-x\|<\|x'\|$}
\end{equation}
(see the right part of Figure \ref{fig:kissing} for an illustration in dimension $2$). Fix a finite family $u_1,\dots,u_{\kappa_8(d)}$ of unit vectors such that the associated family of cones covers $\R^d$.
For each index $i$, set
\[
D_i = \min \{D > 0 : \big| B(0,D) \cap \cone(u_i) \cap S| = M \} \in (0,+\infty].
\]
We then have
\[
\forall x \in S \setminus \Big( \bigcup_i B(0,D_i) \cap \cone(u_i) \Big) , \quad 0 \not\in B(x,r_M(x,S)).
\]
Indeed, fix such an $x$. It belongs to $\cone(u_i)$ for some $i$. Then it does not belong to $B(0,D_i)$ and therefore, by definition of $D_i$ and by \eqref{e:key_stone},
$x$ is strictly closer to all the points of $S \cap B(0,D_i) \cap \cone(u_i)$ than to $0$. This proves the above display. Therefore
\[
\{x \in S : 0 \in B(x,r_M(x,S))\} \subset S \cap \Big( \bigcup_i B(0,D_i) \cap \cone(u_i) \Big)
\]
and the first part of the lemma follows.

From the first part, we deduce the almost sure inequality (i.e., true for any $y \in \R^d \setminus S$)
\[
\sum_{x \in S_f} \1_{B(x,r(x))} \le M\kappa_8(d) \1_{\bigcup_{x \in S_f} B(x,r(x))}.
\]
Integrating the inequality yields the second part of the lemma.
\end{proof}

\begin{figure}
\begin{center}
\begin{tabular}{cp{0.5cm}c}
\includegraphics[width=8cm,height=7cm]{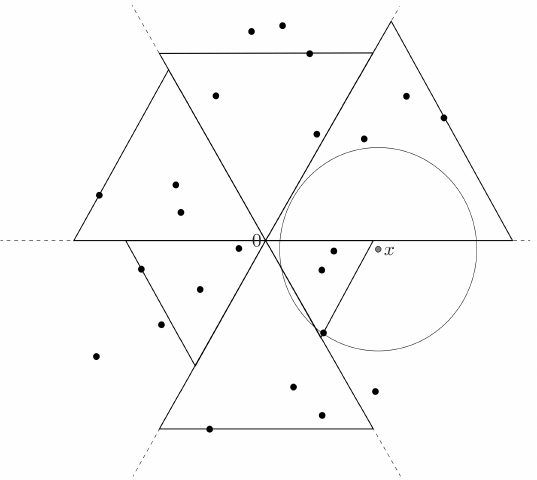} & & \includegraphics[width=6cm,height=5.7cm]{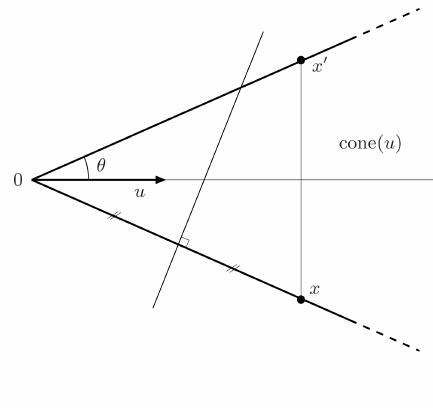}
\end{tabular}
\caption{\label{fig:kissing} To the left: In dimension $d = 2$, the whole space can be covered by $\kappa_8 = 6$ cones with opening angle $\theta = \pi/6$ (the dashed lines indicate their boundaries). Based on these cones, are represented the $6$ triangles whose heights from the apex $0$ are $D_1,\ldots,D_6$. By construction, they all contain exactly $M = 3$ elements of $S$. Moreover, an element $x \in S$, outside the union of those triangles, is depicted in gray: remark that the ball $B(x,r_M(x))$ does not overlap $0$. To the right: Here is the limiting case in (\ref{e:key_stone}) where $x,x' \in \cone(u)$ are such that $\|x\| = \|x'\|$ and $\|x'-x\|$ is maximal. While the opening angle $\theta$ of $\cone(u)$ satisfies $2\theta \leq \pi/3$ (as in the figure), $x'$ will be closer to $x$ than to $0$, meaning that $\|x'-x\| < \|x'\|$.}
\end{center}
\end{figure}

Recall that $\cF_n(T,x,r)$ is the event that Conditions 
\eqref{Fn:pas_trop_de_points}, 
\eqref{Fn:discretisation}, 
\eqref{Fn:pas_loin}, and 
\eqref{Fn:tres_loin} hold,
and that Conditions \eqref{Fn:pas_loin} and \eqref{Fn:tres_loin} are deterministic.
The next lemma is a consequence of Lemma \ref{l:stone}.

\begin{lemma} \label{l:csq_stone} Let $m \in \{1,\dots,n-1\}$, $T \in\cT_m$, $r_0,\dots,r_m \in \N$.
Let $x_1,\dots,x_m$ be distinct non-zero vectors in $\R^d$.
There exists a constant $\kappa_9$ depending only on $d$ and $M$ such that
\[
\tilde \P_{0,x_1,\dots,x_m}[\cF_n(T,x,r)] \le 
\1_{\eqref{Fn:pas_loin}}
\1_{\eqref{Fn:tres_loin}}
\prod_{i=0}^m e^{M-\kappa_9 r_i^d }.
\]
\end{lemma}
\begin{proof} We may assume that the probability $\tilde \P_{0,x_1,\dots,x_m}[\cF_n(T,x,r)]$ is non zero.
This implies two things:
\begin{enumerate}
\item The deterministic Conditions \eqref{Fn:pas_loin} and \eqref{Fn:tres_loin} hold.
\item There exists an infinite and locally finite subset $S \subset \R^d$ containing $S_f = \{0,x_1,\dots,x_m\}$ such that,
for all $v_i$ in $T$, $r_i \le r_M(x_i, S)$. This comes from Condition \eqref{Fn:discretisation}. By Lemma \ref{l:stone} we deduce
\begin{equation}\label{e:stone_consequence}
\cL^d\left(\bigcup_{0 \le i \le m} B(x_i,r_i)\right) \ge \frac 1 {M\kappa_8(d)} \sum_{0 \le i \le m} \cL^d(B(x_i,r_i)).
\end{equation}
\end{enumerate}
We are moreover reduced to prove the existence of a constant $\kappa_9$ depending only on $d$ and $M$ such that
\[
\tilde \P_{0,x_1,\dots,x_m}[\cF_n(T,x,r)] \le e^{(m+1)M}e^{-\kappa_9 \sum_{0 \le i \le m}r_i^d }.
\]
On the left-hand side, this is the probability under the Palm measure $\tilde \P_{0,x_1,\dots,x_m}$ of an event depending on~$\chi$.
By Slivnyak's theorem, this equals the probability of the same event evaluated at $\chi \cup \{0,x_1,\dots,x_m\}$ under $\widetilde\P$.
We consider the latter point of view in this proof.
By \eqref{Fn:pas_trop_de_points}, on the event $\cF_n(T,x,r)$, we have
\[
\left|\big(\chi \cup \{0,x_1,\dots,x_m\} \big) \cap \bigcup_{0 \le i \le m} B(x_i,r_i)\right| \le (m+1)(M+1)
\]
and thus
\[
Z := \left|\chi  \cap \bigcup_{0 \le i \le m} B(x_i,r_i)\right| \le (m+1)M.
\]
The random variable $Z$ is a Poisson random variable whose parameter is lower-bounded, thanks to \eqref{e:stone_consequence}, by
\[
\frac 1 {M\kappa_8(d)} \sum_{0 \le i \le m} \cL^d(B(x_i,r_i)) = \frac {v_d} {M\kappa_8(d)} \sum_{0 \le i \le m}r_i^d
\]
where $v_d$ denotes the Lebesgue measure of the unit ball.
Thus, 
\[
\tilde \P[Z \le (m+1)M] \le e^{(m+1)M} \tilde \E[e^{- Z}]
\]
where $\tilde \E $ is the expectation associated to $\tilde \P$.
The lemma follows.
\end{proof}

\begin{lemma} \label{l:int} Let $m \in \{1,\dots,n-1\}$, $T \in\cT_m$, $r_0,\dots,r_m \in \N$. Then
\[
\int_{(\R^d)^m} \1_{\eqref{Fn:pas_loin}} \d x_1 \cdots \d x_m \le \prod_{0 \le i \le m} \overline v_d^M(r_i+1)^{dM}
\]
where $\overline v_d = \max(v_d,1)$ and $v_d$ is the volume of the unit ball.
\end{lemma}
\begin{proof} 
Consider some node $w_i$ which has at least one child but no grand-child.
The integral over $x_j$ corresponding to children of $w_i$ is 
\[
\left(v_d(r_i+1)^d\right)^{\text{number of children of }w_i} \le \left(\overline v_d(r_i+1)^d\right)^{\text{number of children of }w_i} .
\]
Iterating this argument yields 
\[
\int_{(\R^d)^m} \1_{\eqref{Fn:pas_loin}} \d x_1 \cdots \d x_m \le \prod_{0 \le i \le m} \left(\overline v_d(r_i+1)^d\right)^{\text{number of children of }w_i}.
\]
Upper-bounding the number of children of each node by $M$ gives the result.
\end{proof}

The cardinality of $\cT_m$ is less than the number of planar rooted trees with $m$ edges which is the Catalan number $\frac{1}{m+1}\binom{m}{2m}$.
It is therefore upper-bounded by $4^m$.
Using successively Lemma \ref{l:en_tree}, Lemma \ref{l:csq_stone}, Lemma \ref{l:int}, the bound $|\cT_m|\le 4^m$, the bound
\[
1_{\eqref{Fn:tres_loin}} \le e^{-An} \prod_{i=0}^m e^{M(r_i+1)}
\]
and rearranging we obtain
\begin{align*}
\P[\cE_n] 
 & \le \sum_{m=1}^{n-1} \sum_{T \in \cT_m} \int_{(\R^d)^m} \d x_1 \cdots \d x_m  \sum_{r\in \N^{m+1}}  \tilde \P_{0,x_1,\dots,x_m}[\cF_n(T,x,r)] \\
 & \le \sum_{m=1}^{n-1} \sum_{T \in \cT_m} \int_{(\R^d)^m} \d x_1 \cdots \d x_m  \sum_{r\in \N^{m+1}}  \1_{\eqref{Fn:pas_loin}} \1_{\eqref{Fn:tres_loin}}\prod_{i=0}^m e^{M-\kappa_9 r_i^d }\\
 &  \le \sum_{m=1}^{n-1} \sum_{T \in \cT_m}   \sum_{r\in \N^{m+1}}   \1_{\eqref{Fn:tres_loin}}  \prod_{i=0}^m e^{M-\kappa_9 r_i^d }\overline v_d^M(r_i+1)^{dM}\\
 &  \le \sum_{m=1}^{n-1}    \sum_{r\in \N^{m+1}}   \1_{\eqref{Fn:tres_loin}}  \prod_{i=0}^m 4e^{M-\kappa_9 r_i^d }\overline v_d^M(r_i+1)^{dM}\\
 &  \le  e^{-An}  \sum_{m=1}^{n-1}    \sum_{r\in \N^{m+1}}  \prod_{i=0}^m e^{M(r_i+1)} 4e^{M-\kappa_9 r_i^d }\overline v_d^M(r_i+1)^{dM}\\
 & = e^{-An}  \sum_{m=1}^{n-1}   \left(\sum_{r \ge 0}  4e^{M(r+2)-\kappa_9 r^d }\overline v_d^M(r+1)^{dM}\right)^{m+1}.
\end{align*}
To sum up, there exists a constant $c > 1$ depending only on $d$ and $M$ such that
\begin{align*}
 \P[\cE_n] 
  &  \le  e^{-An}   \sum_{m=1}^{n-1} c^{m+1} = e^{-An} \frac{c^{n+1}}{c-1} .
\end{align*}
Choosing $A$ large enough (depending only on $d$ and $M$), using \eqref{e:decoupage_initial} and Theorem \ref{t:V}, we deduce the existence of $C_3, C_4 >0$
depending only on $d$, $M$ and $\mu (\{M\})$, such that
\[
\forall \ell \ge 0, \P[\L \ge \ell] \le C_3e^{-C_4 \ell}.
\]

\section{Number of steps for discovering the whole range $\V$}
\label{s:S}

The aim of this last section is to prove Theorem \ref{t:S}.
The  proof relies on the construction of a favorable Poissonian environment, in which $(X_n)_{n\geq 0}$ spends a long time before exploring new points. Geometric properties of the set of points that the random walk can reach in one step in this environment play a crucial role. The choice of this favorable Poissonian environment is thus heavily linked to the support of $\mu$. This is the reason why we do not obtain a control on the tail distribution of $\S$ that holds for every $\mu$: we rather consider a fixed simple label distribution, namely $\mu_0 \coloneqq (1-p) \delta_1 + p \delta_2$, for a fixed $p \in (0,1/2)$, and prove that the tail distribution of $\S$ is at least polynomial for this particular label distribution. This example prevents a result analogous to Theorems \ref{t:V} or \ref{t:L} from holding for $\S$.

\subsection{A favorable Poissonian environment}
\label{s:FavorableEnv}

The main part of the proof consists of constructing a favorable Poissonian environment in which $(X_n)_{n\geq 0}$ spends a long time before exploring new points. To this aim, let $e_1$ denote the first vector of the canonical basis of $\R^d$ and let $(z_i)_{i\ge 0}$ be points on the line $\R e_1$ defined by:
\[
\begin{cases} 
z_{3i}\cdot e_1 = 11i \\
z_{3i+1}\cdot e_1 = 11i+2 \\
z_{3i+2}\cdot e_1 = 11i+6.
\end{cases}
\]
Let $B_i=B(z_i,1/5)$ be the ball in $\R^d$ with center $z_i$ and radius $1/5$. Fix $L\ge 1$ and let $\mathcal{C}_L$ be the cylinder 
\[
\mathcal{C}_L \coloneqq \{y=(y_1,\ldots,y_d)\in \R^d : y_1\in[-7, 11L+7], y_k\in [-7,7] \text{ for } k\ge 2\}.
\]
We now consider the following event on the Poisson process $\chi$:
\[
\cH_L \coloneqq \{\forall 0\le i \le 3L, |\chi\cap B_i|=1 \text{ and } |\chi\cap \mathcal{C}_L|=3L+1\},
\]
{\it i.e.}, exactly one point of $\chi$ falls in each ball $B_i$ for $i\le 3L$, and no point of $\chi$ falls in $\mathcal{C}_L$ outside these balls $B_i$ (see Figure \ref{fig:bonenvironnement} for an illustration). If $\cH_L$ occurs, let $x_i$, for $0\le i \le 3L$, denote the unique point of $\chi$ in the ball $B_i$ (since we work with the Palm measure, note that $x_0=0$). Our first lemma lower-bounds the probability that such an environment occurs.

\begin{figure}
    \centering
\includegraphics[width=13cm]{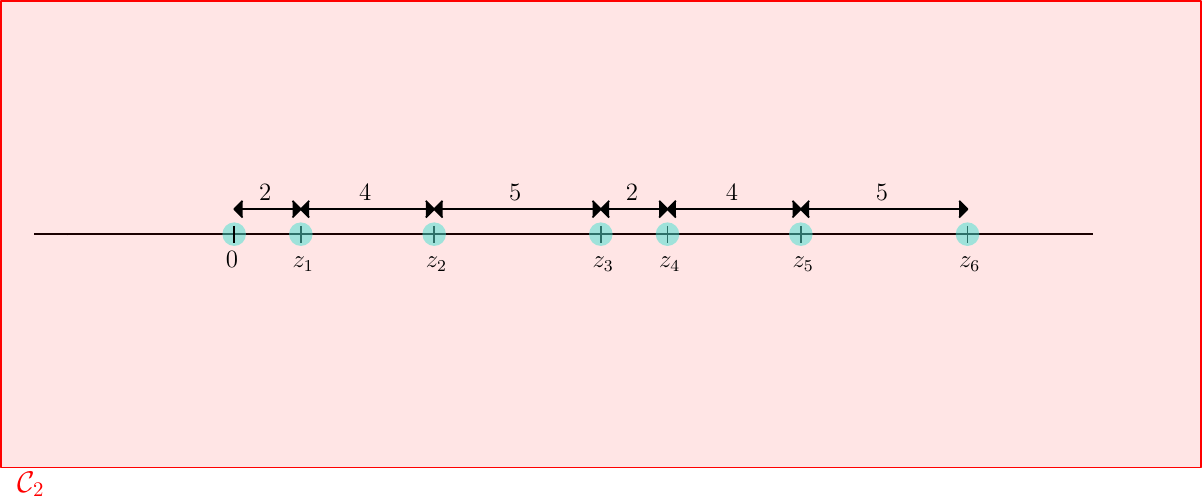}
\caption{\small{Illustration of required environment for $L=2$. We impose that exactly one point $x_i$ of $\chi$ falls in each blue ball and no point of $\chi$ falls in the red region.  }} \label{fig:bonenvironnement}
\end{figure}

\begin{lemma}
\label{l:probaH}
There exist constants $\kappa_{10},\kappa_{11} \in (0,+\infty)$, depending only on $d$, such that
\[
\forall L\ge 1 \qquad \P ( \cH_L ) \geq \kappa_{10} e^{-\kappa_{11}L} ~.
\]
\end{lemma}

\begin{proof}
Let $\tilde{\mathcal{C}}_L \coloneqq \mathcal{C}_L \setminus \left(\{0\}\cup \bigcup_{i=1}^{3L}B_i\right)$. Since the sets $B_i$ (for $1\le i \le 3L$) and $\tilde{\mathcal{C}}_L$ are pairwise disjoint, the independence properties of the Poisson process imply:
\begin{align*}
\P ( \cH_L ) &= \P(\forall 1\le i \le 3L, |\chi\cap B_i|=1) \cdot \P(|\chi\cap \tilde{\mathcal{C}}_L|=0) \\
&= \P( |\chi\cap B_1|=1)^{3L} \exp(-\operatorname{Vol}(\tilde{\mathcal{C}}_L)) \\
&\ge \P( |\chi\cap B_1|=1)^{3L} \exp(-\operatorname{Vol}(\mathcal{C}_L)).
\end{align*}
Since $\operatorname{Vol}(\mathcal{C}_L)$ grows linearly with $L$, we obtain $\P(\cH_L) \ge \kappa_{10} e^{-\kappa_{11}L}$ for some constants $\kappa_{10},\kappa_{11} > 0$ depending on $d$.
\end{proof}

If $\mu = \mu_0 \coloneqq (1-p) \delta_1 + p \delta_2$ for $p \in (0,1)$, then on the event $\cH_L$, the process $(X_n)_{n\ge 0}$ can be described using a specific Markov chain.

\begin{prop}\label{prop:defY}
Let $(Y^L_n)_{n\ge 0}$ be the Markov chain on $\{0,1,\ldots,3L\}$ with $Y^L_0=0$ and transition probabilities:
\[
\P(Y^L_{n+1}=j \mid Y^L_n=i) = \begin{cases} 
(1-p)\mathbf{1}_{j=1}+p\mathbf{1}_{j=2} & \text{if } i=0 \\
(1-p)\mathbf{1}_{j=i+1}+p\mathbf{1}_{j=i-1} & \text{if } i\equiv 0 \pmod 3 \text{ and } 0<i<3L \\
(1-p)\mathbf{1}_{j=i-1}+p\mathbf{1}_{j=i+1}  & \text{if } i\not\equiv 0 \pmod 3 \text{ and } 0<i<3L \\
(1-p)\mathbf{1}_{j=3L-1}+p\mathbf{1}_{j=3L-2}  & \text{if } i=3L.
\end{cases}
\]
If $\mu = \mu_0$, then on the event $\cH_L$, the walk $(X_n)_{n\ge 0}$ has the law of $(x_{Y^L_n})_{n\ge 0}$. 
\end{prop}

\begin{proof}
If $\cH_L$ occurs, the neighbors of $x_i$ (the points in $\chi$ closest to $x_i$) are determined by the geometry of the balls $B_i$. For $1 < i < 3L$, if $i\equiv 0 \pmod 3$, then $x_{i+1}$ is the closest point and $x_{i-1}$ is the second closest. Thus $v_1(x_i)=x_{i+1}$ and $v_2(x_i)=x_{i-1}$. Conversely, if $i\equiv 1$ or $2 \pmod 3$, then $x_{i-1}$ is closer than $x_{i+1}$, so $v_1(x_i)=x_{i-1}$ and $v_2(x_i)=x_{i+1}$. The boundary cases $i=0$ and $i=3L$ follow similarly. Hence, $(X_n)_{n\ge 0}$ follows the transitions of $Y^L_n$ jumping between the points $\{x_i\}$.
\end{proof}

We now use a lemma on the exit time of $(Y^L_n)_{n\ge 0}$. Its proof, based on classical gambler's ruin estimates in biased environments, is postponed to the end of this section.

\begin{lemma}\label{l:MA2} 
Fix $p\in (0,1/2)$ and let $(Y^L_n)_{n\ge 0}$ be the Markov chain defined in Proposition \ref{prop:defY}. Let $T_L \coloneqq \inf\{n\ge 0 : Y^L_n=3L\}$. Then $T_L < \infty$ a.s. and there exist constants $\kappa_{12},\kappa_{13}>0$, depending only on $p$, such that
\[ \forall L\ge 1 \qquad \P ( T_L > e^{\kappa_{12} L} ) \geq \kappa_{13}. \]
\end{lemma}

\begin{proof}[Proof of Theorem \ref{t:S}] 
Recall that $\S$ denotes the last time $(X_n)_{n\ge 0}$ visits a new site. Fix $L\ge 1$ and define
$$\sigma_L:=
\begin{cases}
 \inf\{n\ge 0, X_n=x_{3L}\}\ & \text{if  $\mathcal{H}_L$ occurs}. \\
0  & \text{otherwise}.
\end{cases}$$
 Since, on the event $\mathcal{H}_L$,  $x_{3L}$ is ultimately visited by $(X_n)_{n\ge 0}$ and  has not been visited before $T_L$, we have $\S \ge \sigma_L$ a.s. Moreover, using Proposition \ref{prop:defY} and Lemma \ref{l:MA2}, we have
\[ \P( \sigma_L > e^{\kappa_{12} L} \mid \cH_L) = \P ( T_L > e^{\kappa_{12} L} ) \geq \kappa_{13}. \]
In view of Lemma \ref{l:probaH}, we deduce
\[ \P (\S \geq e^{\kappa_{12} L}) \geq \P(\cH_L) \P( \sigma_L > e^{\kappa_{12} L} \mid \cH_L) \ge \kappa_{10} \kappa_{13} e^{-\kappa_{11} L}. \]
Finally, for $n\ge 1$, choosing $L$ such that $e^{\kappa_{12} (L-1)} < n \le e^{\kappa_{12} L}$, we obtain:
\[ \P (\S \geq n) \ge \P (\S \geq e^{\kappa_{12} L}) \ge \kappa_{10} \kappa_{13} e^{-\kappa_{11} L} = \kappa_{10} \kappa_{13} e^{-\kappa_{11} } (e^{\kappa_{12}(L-1)})^{-\kappa_{11}/\kappa_{12}} \ge \kappa_{10} \kappa_{13} e^{-\kappa_{11}} n^{-\kappa_{11}/\kappa_{12}} \]
as claimed in Theorem \ref{t:S}.
\end{proof}

\subsection{Proof of Lemma \ref{l:MA2} on random walks}
Lemma \ref{l:MA2} is a direct consequence of the following two lemmas.

\begin{lemma}\label{lemm:defZ}
Fix $p\in (0,1)$ and let $(Z_n)_{n\ge 0}$ be the Markov chain on $\mathbb{Z}$ starting from $0$ with transition probabilities:
\[
\P(Z_{n+1}=j \mid Z_n=i) = \begin{cases} 
(1-p)\mathbf{1}_{j=1}+p\mathbf{1}_{j=-1} & \text{if } i= 0 \\
(1-p)\mathbf{1}_{j=i-1}+p\mathbf{1}_{j=i+1}  & \text{if } i\neq 0. \end{cases}
\]
Let $U \coloneqq \inf\{n\ge 0 : |Z_n|=3\}$. Then 
\[ \P(Z_U=3) = 1 - \P(Z_U=-3) = p. \]
\end{lemma}

\begin{lemma}\label{lemm:defZtilde}
Fix $p\in (0,1/2)$ and let $(\tilde{Z}_n)_{n\ge 0}$ be the biased random walk on $\mathbb{N}$ with bias $2p-1$, starting from $0$ and reflected at $0$, i.e., with transition probabilities:
\[
\P(\tilde{Z}_{n+1}=j \mid \tilde{Z}_n=i) = \begin{cases} 
\mathbf{1}_{j=1} & \text{if } i= 0 \\
(1-p)\mathbf{1}_{j=i-1}+p\mathbf{1}_{j=i+1}  & \text{if } i > 0.\end{cases}
\]
For $L\ge 1$, let $\tilde{U}_L \coloneqq \inf\{n\ge 0 : \tilde{Z}_n=L\}$. Then there exist constants $\kappa_{12}, \kappa_{13} > 0$ (depending on $p$) such that for any $L\ge 1$:
\[ \P(\tilde{U}_L \ge e^{\kappa_{12} L}) \ge \kappa_{13}. \]
\end{lemma}

The proofs of these two lemmas contain no original elements and are likely already present in the literature. However, as we were unable to find a precise reference, we provide self-contained proofs here for the sake of completeness.

\begin{proof}[Proof of Lemma \ref{lemm:defZ}] 
Let $f(i) \coloneqq \P(Z_U=3 \mid Z_0=i)$. Then $f$ satisfies the following discrete Dirichlet problem:
\[
f(i) = \begin{cases} 
(1-p)f(1)+pf(-1) & \text{if } i= 0 \\
(1-p)f(i-1)+pf(i+1) & \text{if } i\in \{-2,-1,1,2\}\\
1 & \text{if } i=3\\
0 & \text{if } i=-3. \end{cases}
\]
A straightforward calculation shows that:
\[
(1-p+p^2)f(i) = \begin{cases} 
p^3 & \text{if } i=-2\\
p^2 & \text{if } i=-1\\
p(1-p+p^2) & \text{if } i=0\\
p & \text{if } i=1\\
p(2-2p+p^2) & \text{if } i=2 .
 \end{cases}
\]
In particular, $f(0)=p$.
\end{proof}

\begin{proof}[Proof of Lemma \ref{lemm:defZtilde}]
Note that the bias $2p-1$ of the random walk $(\tilde{Z}_n)_{n\ge0}$ is strictly negative since $p < 1/2$. For $L\ge 2$, let $\tau_L \coloneqq \inf\{n > 0 : \tilde{Z}_n \in \{0,L\}\}$. An application of the Optional Stopping Theorem to the martingale $M_n = (\frac{1-p}{p})^{\tilde{Z}_n}$ yields the classical ruin probability:
 \[\forall L\ge 2, \qquad \P(\tilde{Z}_{\tau_L}=L \mid \tilde{Z}_0=1) \le \exp(-\kappa_{12} L) \]
for some $\kappa_{12} > 0$ depending on $p$. If $H_L$ denotes the number of visits to $0$ by $\tilde{Z}$ before reaching $L$, then $H_L$ follows a geometric distribution with parameter $q_L \coloneqq \P(\tilde{Z}_{\tau_L}=L \mid \tilde{Z}_0=1)$. Clearly, $\tilde{U}_L \ge H_L$. Thus, for any $k \ge 1$:
\[ \P(\tilde{U}_L \ge k) \ge \P(H_L \ge k) = (1-q_L)^{k-1} \ge (1-e^{-\kappa_{12} L})^{k-1}. \]
Setting $k = \lceil e^{\kappa_{12} L} \rceil$, we obtain:
\[ \P(\tilde{U}_L \ge e^{\kappa_{12} L}) \ge (1-e^{-\kappa_{12} L})^{e^{\kappa_{12} L}} \ge \kappa_{13}, \]
where $\kappa_{13} > 0$ is a constant, since $(1-1/x)^x \to 1/e$ as $x \to \infty$.
\end{proof}

\begin{proof}[Proof of Lemma \ref{l:MA2}]
Consider the Markov chain $(Y^L_n)_{n\ge 0}$ only at the times when it visits states of the form $3i$ for $0\le i\le L$. Let $(\tilde{Y}^L_k)_{k\ge 0}$ denote this sampled process with values in $\{0,3,6,\dots,3L\}$. This process is itself a Markov chain. According to Proposition \ref{prop:defY} and Lemma \ref{lemm:defZ}, its transition probabilities for $0 < i < L$ satisfy:
\[ \P(\tilde{Y}_{k+1}=3j \mid \tilde{Y}_k=3i) = (1-p)\mathbf{1}_{j=i-1}+p\mathbf{1}_{j=i+1}. \]
At the boundary $i=0$, the chain moves to $3$ or stays at $0$ with probabilities that favor $0$. Thus, $(\tilde{Y}^L_k)_{k\ge 0}$ is stochastically dominated by the reflected biased random walk $(3\tilde{Z}_k)_{k\ge 0}$ from Lemma \ref{lemm:defZtilde}. Let $\tilde{T}_L \coloneqq \inf\{k \ge 0 : \tilde{Y}_k^L = 3L\}$. By Lemma \ref{lemm:defZtilde}, we have:
\[ \P(\tilde{T}_L \ge e^{\kappa_{12} L}) \ge \P(\tilde{U}_L \ge e^{\kappa_{12} L}) \ge \kappa_{13}. \]
Since each step of $\tilde{Y}^L$ corresponds to at least one step of $Y^L$, we have $T_L \ge \tilde{T}_L$, which concludes the proof.
\end{proof}

\end{document}